\documentclass[a4paper,11pt,leqno]{article}
\usepackage{color}
\usepackage{amsmath}
\catcode`\@=11 \@addtoreset{equation}{section}

\catcode`\@=12

\usepackage{amssymb}
\usepackage{amsfonts}
\usepackage{xcolor}
\usepackage{graphics}
\usepackage{epsfig,psfrag,graphicx}
\usepackage{amssymb}
\usepackage{eepic,epic}
\usepackage{dsfont}

\usepackage{enumerate}

\usepackage{epsfig} 
\usepackage{amsmath} 
\usepackage{amsthm} 
\usepackage{amssymb} 
\usepackage{amsbsy} 
\usepackage{bbm}
\usepackage{verbatim}

\theoremstyle{definition}
\newtheorem{definition}{Definition}[section]

\theoremstyle{plain}
\newtheorem{theorem}[definition]{Theorem}
\newtheorem{proposition}[definition]{Proposition}
\newtheorem{corollary}[definition]{Corollary}
\newtheorem{lemma}[definition]{Lemma}
\newtheorem{remark}[definition]{Remark}

\newcommand{\nc}{\newcommand}
\nc{\weak}{\rightharpoonup}
\nc{\weakstar}{\stackrel{\ast}{\rightharpoonup}} 
\renewcommand{\div}{{{\mathrm{div}}_x}\,}
\newcommand{\vrho}{\varrho}
\nc{\modular}[1]{{\stackrel{ #1}{\longrightarrow\,}}}

\def\vec#1{\boldsymbol{#1}}

\newcommand{\anl}{a_\veps^{n+l}}
\newcommand{\an}{a_\veps^{n+1}}
\newcommand{\vu}{\vec{u}}
\newcommand{\vr}		{\vrho}
\newcommand{\vre}		{\vr_\varepsilon}
\newcommand{\vrez}	{\vr_{0,\varepsilon}}

\newcommand{\ue}		{\vec{u}_\varepsilon}
\newcommand{\uez}		{\vec{u}_{0,\varepsilon}}
\newcommand{\ep}		{\varepsilon}

\newcommand{\ds}		{\,{\rm d}s}
\newcommand{\dx}		{\,{\rm d}x}
\newcommand{\dt}		{\, {\rm d}t}
\newcommand{\detau}		{\, {\rm d}\tau}
\newcommand{\dxdt}	{\, {\rm d}x{\rm d}t}

\newcommand{\mc}{\mathcal}

\newcommand{\veps}{\varepsilon}

\newcommand{\what}{\widehat}

\newcommand{\vphi}{\varphi}
\newcommand{\oline}{\overline}

\newcommand{\R}{\mathbb{R}}

\newcommand{\N}{\mathbb{N}}
\newcommand{\Z}{\mathbb{Z}}

\newcommand{\TT}{\mathbb{T}}

\renewcommand{\div}{{\rm div}\,}
\newcommand{\curl}{{\rm curl}\,}

\newcommand{\Supp}{{\rm Supp}\,}

\allowdisplaybreaks

\def\d{\partial}
\def\div{{\rm div}\,}

\title{\LARGE \bf{Fast rotation limit for the 2-D\\
non-homogeneous incompressible Euler equations
}}
\author{ 
\textsl{Gabriele Sbaiz}$\,^{1,2}$
 \vspace{.2cm} \\
\footnotesize{$\,^1\;$ \textsc{Universit\`a degli Studi di Trieste}, \textit{Dipartimento di Matematica e Geoscienze},} \\
\footnotesize{Via Valerio 12/1, 34127 Trieste, Italy} \vspace{0.1cm} \\
\footnotesize{$\,^2\;$ \textsc{Univ. Lyon, Universit\'e Claude Bernard Lyon 1}, CNRS UMR 5208, \textit{Institut Camille Jordan},} \\
\footnotesize{43 blvd. du 11 novembre 1918, F-69622 Villeurbanne cedex, France} \vspace{0.1cm} \\
\vspace{.2cm} \\
\footnotesize{
\ttfamily{gabriele.sbaiz@phd.units.it}} \vspace{.1cm}
}

\date{\small \today}

\begin{document}
\maketitle

\abstract{In the present paper, we study the fast rotation limit for the density-dependent incompressible Euler equations in two space dimensions with the presence of the Coriolis force. In the case when the initial densities are small perturbation of a constant profile, we show the convergence of solutions towards the solutions of a quasi-homogeneous incompressible Euler system. The proof relies on a combination of uniform estimates in high regularity norms with a compensated compactness argument for passing to the limit. This technique allows us to treat the case of ill-prepared initial data.}

\paragraph*{\small 2010 Mathematics Subject Classification:}{\footnotesize 35Q86 
(primary);
35B25, 
76U05, 
35B40, 
76M45 
(secondary).}


\section{Introduction}
In this paper, we are interested in describing the evolution of the density and the velocity field of a fluid in a rotational framework. The density $\vrho=\vrho (t,x)$ is a scalar function belonging to $\R_+$ and $\vec u=\vec u(t,x)$ represents the velocity field of the fluid on $\R^2$. The choice of the $\R^2$ plane is motivated by the fact that the $2$-D setting is relevant for fluids in a fast rotation regime, like currents in the oceans under the Earth's rotational effects. Indeed, the motion of a $3$-D highly rotating fluid is, in a first approximation, planar: this property is the so-called \textit{Taylor-Proudamn theorem} (see Chapter 7 of \cite{C-R}, Chapter 2 of \cite{Ped} and Chapter 2 of \cite{Val} for details in this respect). Then, the mathematical model we are going to consider, describes the dynamics of flows on large scales which occurs in nature (i.e. the so-called \textit{geophysical flows}). In our model we keep three characteristics which are relevant from the physical side: the fluid is supposed to be \textit{non-homogeneous} (we deal with variations of the density), \textit{incompressible} (the volumes are preserved along the motion) and \textit{inviscid} (the viscosity effects are neglected). Therefore, the system describing the 2-D evolution of the fluid, reads
\begin{equation}\label{full Euler}
\begin{cases}
\d_t \vrho +\div (\vrho \vu)=0\\
\d_t (\vrho \vu)+\div (\vrho \vu \otimes \vu)+ \frac{1}{\mathfrak{Ro}}\vrho \vu^\perp+\nabla P=0\\
\div \vu =0
\end{cases}
\end{equation}
in the domain $\Omega =\R^2$. 

The pressure term $\nabla P$, where $P=P(t,x)\in \R$, represents the Lagrangian multiplier associated to the divergence-free constraint on the velocity field. 
In addition, the rotational effects (due to the rotation of the Earth) are translated in the system \eqref{full Euler} by the presence of the Coriolis force $\frac{1}{\mathfrak{Ro}}\, \vrho \vu^\perp$, where $\vu^\perp:=(-u_2, u_1)$ is the rotation of angle $\pi/2$ of the velocity field $\vu=(u_1,u_2)$. This simple form for the Coriolis term is a physically well-justified approximation of that force at mid-latitudes: we consider the motion far from the poles and far from the equator (see again Chapter 7 of \cite{C-R}, Chapter 1 of \cite{Ped} and Chapter 2 of \cite{Val} for useful insight). Finally, the adimensional number $\mathfrak{Ro}$ in \eqref{full Euler} is the so-called \textit{Rossby number}. It represents the inverse of the rotation speed: choosing the Rossby number small will mean considering fast rotational effects in the dynamics. Specifically, we introduce the following scaling: given $\veps \in \; ]0,1]$ we take     
\begin{equation}\label{scaling}
 \mathfrak{Ro}=\veps .
\end{equation}
The main scope of our analysis will be to study the asymptotic behaviour of the system \eqref{full Euler} when $\veps\rightarrow 0$. We refer to \cite{C-D-G-G} for an overview of the broad literature in the context of homogeneous rotating fluids (see \cite{B-M-N_EMF} and \cite{B-M-N_AA} for the pioneering studies). 

Similar problems to the one presented in \eqref{full Euler} have been studied by several authors, who in different ways inspected the well-posedness issues and the asymptotic analysis of models for geophysical flows. For example, in the context of compressible fluids,
we refer to \cite{B-D}, \cite{G_2008}, \cite{G-SR_Mem} for the first works on $2$-D viscous shallow water models (see also \cite{D_ARMA}, \cite{D-M}, \cite{D-M-S} for the inviscid case), to \cite{F-G-GV-N}, \cite{F-G-N}, \cite{F-N_AMPA}, \cite{F-N_CPDE} for the barotropic Navier-Stokes system and to \cite{N-S_DCDS} for weakly compressible and inviscid fluids (see also \cite{F-N} for other singular limits in thermodynamics of viscous fluids).
In the compressible case, the fact that  the pressure is a given function of the density implies a double advantage in the analysis: on the one hand, one can recover good uniform bounds for the oscillations (from the reference state) of the density; on the other hand, at the limit, one disposes of a stream-function relation between the densities and the velocities. 

On the contrary, although the incompressibility condition is physically well-justified for the geophysical fluids, only few studies tackle this case. We refer to \cite{Fan-G}, in which Fanelli and Gallagher have studied the fast rotation limit for viscous incompressible fluids with variable density. In the case when the initial density is a small perturbation of a constant state (the so-called \textit{slightly non-homogeneous} case), they proved convergence to the quasi-homogeneous type system. Instead, for general non-homogeneous fluids (the so-called \textit{fully non-homogeneous} case), they have showed that the limit dynamics is described in terms of the vorticity and the density oscillation function, since they lack enough regularity to prove convergence on the momentum equation itself (see more details below).

We have also to mention \cite{C-F_RWA}, where the authors rigorously prove the convergence of the ideal magnetohydrodynamics (MHD) equations towards a quasi-homogeneous type system (see also \cite{C-F_Nonlin} in this respect). Their method relies on a relative entropy inequality for the primitive system that allows to treat also the inviscid limit but requires well-prepared initial data. 


In the present paper, we tackle the asymptotic analysis (for $\veps\rightarrow 0$) in the case of density-dependent Euler system in the \textit{slightly non-homogeneous} context, i.e. when the initial density is a small perturbation of order $\veps$ of a constant profile (say $\overline{\vrho}=1$). These small perturbations around a constant reference state are physically justified by the so-called \textit{Boussinesq approximation} (see e.g. Chapter 3 of \cite{C-R} or Chapter 1 of \cite{Maj} in this respect). As a matter of fact, since the constant state $\overline{\vrho}=1$ is transported by a divergence-free vector field, the density can be written as $\vrho_\veps=1+\veps R_\veps$ at any time (provided this is true at $t=0$), where one can state good uniform bounds on $R_\veps$. We also point out that in the momentum equation of \eqref{full Euler}, with the scaling introduced in \eqref{scaling}, the Coriolis term can be rewritten as
\begin{equation}\label{Coriolis}
 \frac{1}{\veps}\vrho_\veps \vec u_\veps^\perp=\frac{1}{\veps}\vec u_\veps^\perp+R_\veps\vec u_\veps^\perp \, .
\end{equation}
We notice that, thanks to the incompressibility condition, the former term on the right-hand side of \eqref{Coriolis} is actually a gradient: it can be ``absorbed'' into the pressure term, which must scale as $1/\veps$. In fact, the only force that can compensate the effect of fast rotation in system \eqref{full Euler} is, at geophysical scale, the pressure term: i.e. we can write $\nabla P_\veps= (1/\veps) \, \nabla \Pi_\veps$.

Let us point out that the \textit{fully non-homogeneous} case (where the initial density is a perturbation of an arbitrary state) is out of our study. This case is more involved and new technical troubles arise in the well-posedness analysis and in the asymptotic inspection. Indeed, as already highlighted in \cite{Fan-G} for the Navier-Stokes-Coriolis system, the limit dynamics is described by an underdetermined system which mixes the vorticity and the density fluctuations.
In order to depict the full limit dynamics (where the limit density variations and the limit velocities are decoupled), one had to assume stronger \textit{a priori} bounds than the ones which could be obtained by classical energy estimates. Nonetheless, the higher regularity involved is \textit{not} propagated uniformly in $\veps$ in general, due to the presence of the Coriolis term. 
In particular, the structure of the Coriolis term is more complicated  than the one in \eqref{Coriolis} above, since one has $\vrho_{\veps}=\oline \vrho+ \veps \sigma_\veps$ (with $\sigma_\veps$'s the fluctuations), if at the initial time we assume $\vrho_{0, \veps}=\oline \vrho+ \veps R_{0,\veps}$, where $\oline \vrho$ represents the arbitrary reference state. At this point, if one plugs the previous decomposition of $\vrho_\veps$ in \eqref{Coriolis}, a term of the form $(1/\veps)\, \oline \vrho \ue^\perp$ appears: this term is a source of troubles in order to propagate the $H^s$ estimates.

Equivalently, if one ties to divide the momentum equation in \eqref{full Euler} by the density $\vrho_\veps$, then the previous issue is only translated on the analysis of the pressure term, which becomes $1/(\veps \vrho_\veps)\, \nabla \Pi_\veps$.

In light of all the foregoing discussion, let us now point out the main difficulties arising in our work.

First of all, our model is an inviscid and hyperbolic type system for which we can expect \textit{no} smoothing effects and \textit{no} gain of regularity. For that reason, it is natural to look at equations in \eqref{full Euler} in a regular framework like the $H^s$ spaces with $s>2$. The Sobolev spaces $H^s(\R^2)$, for $s>2$, are in fact embedded in the space $W^{1,\infty}$ of globally Lipschitz functions: this is a minimal requirement to preserve the initial regularity (see e.g. Chapter 3 of \cite{B-C-D} and also \cite{D_JDE}, \cite{D-F_JMPA} for a broad discussion on this topic). 
As a matter of fact, all the Besov spaces $B^s_{p,r}(\R^d)$ which are embedded in $W^{1,\infty}(\R^d)$, a fact that occurs for $(s,p,r)\in \R\times [1,+\infty]^2$ such that 
\begin{equation}\label{Lip_assumption}
s>1+\frac{d}{p} \quad \quad \quad \text{or}\quad \quad \quad s=1+\frac{d}{p} \quad \text{and}\quad r=1\, ,
\end{equation}
are good candidates for the well-posedness analysis. However, the choice of working in $H^s\equiv B^s_{2,2}$ is dictated by the presence of the Coriolis force: we will deeply exploit the antisymmetry of this singular term. 

Moreover, the fluid is assumed to be incompressible, so that the pressure term is just a Lagrangian multiplier and does \textit{not} give any information on the density, unlike in the compressible case. In addition, due to the non-homogeneity, the analysis of the gradient of the pressure term is much more involved since we have to deal with an elliptic equation with \textit{non-constant} coefficients, namely 
\begin{equation}\label{elliptic_eq}
-\div (A \, \nabla P)=\div F \quad \text{where}\quad \div F:=\div \left(\vu \cdot \nabla \vu+ \frac{1}{\mathfrak{Ro}} \vu^\perp \right)\quad \text{and}\quad A:=1/\vrho \, .
\end{equation}
The main difficulty is to get appropriate \textit{uniform} bounds (with respect to the rotation parameter) for the pressure term in the regular framework we will consider 
(we refer to \cite{D_JDE} and \cite{D-F_JMPA} for more details). 

Once we have analysed the pressure term, we will show the local well-posedness for system \eqref{full Euler} in the $H^s$ setting (see Theorem \ref{W-P_fullE} below). It is worth to notice that, in Theorem \ref{W-P_fullE} below, all the estimates are \textit{uniform} with respect to the rotation parameter and, in addition, we have that the time of existence is independent of $\veps$. 

With the local well-posedness result at the hand, we perform the fast rotation limit for general \textit{ill-prepared} initial data. We will show the convergence of system \eqref{full Euler} towards what we call quasi-homogeneous incompressible Euler system 
\begin{equation}\label{Q-H_E_intro}
\begin{cases}
\d_t R+\div (R\vu)=0 \\
\d_t \vec u+\div \left(\vec{u}\otimes\vec{u}\right)+R\vu^\perp+ \nabla \Pi =0 \\
\div \vec u\,=\,0\,, 
\end{cases}
\end{equation}  
where $R$ represents the limit of fluctuations $R_\veps$ (see Theorem \ref{thm:limit_dynamics} for details). 
We also point out that in the momentum equation of \eqref{Q-H_E_intro} a non-linear term of lower order (i.e. $R\vec u^\perp$) appears: it is a sort of remainder in the convergence for the Coriolis term, recasted as in \eqref{Coriolis}.

Passing to the limit in the momentum equation of \eqref{full Euler} is no more evident, although we are in the $H^s$ framework: the Coriolis term is responsible for strong in time oscillations of solutions (the so-called \textit{Poincar\'e waves}) which may prevent the convergence of the convective term towards the one of \eqref{Q-H_E_intro}. To overcome this issue, we employ an approach based on a compensated compactness argument (see e.g. \cite{Fan-G} and reference therein, in the case of viscous fluids). This technique was firstly applied to the barotropic Navier-Stokes equations by Lions and Masmoudi in \cite{L-M} and later developed in the fast rotation, incompressible and homogeneous case by Gallagher and Saint-Raymond in \cite{G-SR_2006}. The strategy consists in making use of the algebraic structure hidden behind the system (recasted as a wave system) to reveal strong convergence properties for special quantities: in our case, $\gamma_\veps:=\curl (\vrho_\veps \ue)$ (see Section \ref{ss:wave_system} below).
We refer also to \cite{C-F_RWA}, for a different approach based on  relative entropy inequalities for the primitive equations to prove the convergence towards the limit system.

Now, once the limit system is rigorously depicted, one could address its well-posedness issue: it is worth noticing that system \eqref{Q-H_E_intro} is \textit{not} globally well-posed even in two dimensions. However, roughly speaking, for $R_0$ small enough, the system \eqref{Q-H_E_intro} is ``close'' to the $2$-D homogeneous and incompressible Euler system, for which it is well-known the global well-posedness. Thus, it is natural to wonder if there exists an ``asymptotically global'' well-posedness result in the spirit of \cite{D-F_JMPA} and \cite{C-F_sub}: for small initial fluctuations $R_0$, the quasi-homogeneous system \eqref{Q-H_E_intro} behaves like the standard Euler equations and the lifespan of its solutions tends to infinity. In particular, as already shown in \cite{C-F_sub} for the quasi-homogeneous ideal MHD system (see also references therein) the lifespan goes as  
\begin{equation}\label{lifespan_Q-H}
T_\delta^\ast \sim \log \log \frac{1}{\delta}	\,, 
\end{equation}
where $\delta>0$ is the size of the initial fluctuations (see Theorem \ref{thm:well-posedness_Q-H-Euler} below). 

The result for the time of existence of solutions to \eqref{Q-H_E_intro} pushes our attention to the study of the lifespan of solutions to the primitive system \eqref{full Euler}.
For the $3$-D \textit{homogeneous} Euler system with the Coriolis force, Dutrifoy in \cite{Dut} has proved that the lifespan of solutions tends to infinity in the fast rotation regime (see also \cite{Gall}, \cite{Cha} and \cite{Scro}, where the authors inspected the lifespan of solutions in the context of viscous homogeneous fluids). For system \eqref{full Euler} it is not clear to us how to find similar stabilization effects (due to the Coriolis term), in order to improve the lifespan of the solutions: for instance to show that $T_\veps^\ast\rightarrow +\infty$ when $\veps\rightarrow 0$. Nevertheless, independently of the rotational effects, we are able to state an ``asymptotically global'' well-posedness result in the regime of \textit{small} oscillations, in the sense of \eqref{lifespan_Q-H}: namely, when the size of the initial fluctuation $R_{0,\veps}$ is small enough, of size $\delta >0$, the lifespan $T^\ast_\veps$ of the corresponding solution to system \eqref{full Euler} can be bounded from below by $T^\ast_\veps\geq T^\ast(\delta)$, with $T^\ast (\delta)\rightarrow +\infty$ when $\delta\rightarrow 0$
(see also \cite{D-F_JMPA} for a density-depend fluid in the absence of the Coriolis force). As an immediate corollary of the previous lower bound, if we consider the initial densities of the form $\vrho_{0,\veps}=1+\veps^{1+\alpha}R_{0,\veps}$ with $\alpha >0$, then we get $T^\ast_\veps\sim \log \log (1/\veps)$.
We refer to Theorem \ref{W-P_fullE} below for the precise statement.

At this point, let us sketch the main steps to show \eqref{lifespan_Q-H} for the primitive system \eqref{full Euler}.

The key point in the proof of \eqref{lifespan_Q-H} is to study the lifespan of solutions in critical Besov spaces. In those spaces, we can take advantage of the fact that, when $s=0$, the $B^0_{p,r}$ norm of solutions can be bounded \textit{linearly} with respect to the Lipschitz norm of the velocity, rather than exponentially (see the works \cite{Vis} by Vishik and \cite{H-K} by Hmidi and Keraani). Since the triplet $(s,p,r)$ has to satisfy \eqref{Lip_assumption}, the lowest regularity Besov space we can reach is $B^1_{\infty,1}$. Then if $\vec u$ belongs to $B^1_{\infty,1}$, the vorticity $\omega := -\d_2 u_1+\d_1 u_2$ has the desired regularity to apply the quoted improved estimates by Hmidi-Keraani and Vishik (see Theorem \ref{thm:improved_est_transport} in the Appendix). 
Analysing the vorticity formulation of the system, we discover that the $\curl$ operator cancels the singular effects produced by the Coriolis force (in this respect, see equation \eqref{curl_eq} below). That cancellation is not apparent, since the skew-symmetric property of the Coriolis term is out of use in the critical framework considered.

Finally, we need a continuation criterion (in the spirit of Baele-Kato-Majda criterion, see \cite{B-K-M}) which guarantees that we can ``measure'' the lifespan of solutions indistinctly in the space of lowest regularity index, namely $s=r=1$ and $p=+\infty$. That criterion is valid under the assumptions that 
$$
\int_0^{T}  \big\| \nabla \vec u(t) \big\|_{L^\infty}  \dt < +\infty\qquad \text{with}\qquad T<+\infty\, .
$$

We refer to Subsection \ref{ss:cont_criterion+consequences} below for more detailed consequences of the previous continuation criterion.

\medbreak
Let us now give a more precise overview of the contents of the paper.
In the next section, we collect our assumptions and we state our main results. 
In Section \ref{s:well-posedness_original_problem}, we investigate the well-posedness issues in the Sobolev spaces $H^s$ for any $s>2$.
In Section \ref{s:sing-pert}, we study the singular perturbation problem, establishing constraints that the limit points have to satisfy and proving the convergence to the quasi-homogeneous Euler system thanks to a \textit{compensated compactness} technique. 
In Section \ref{s:well-posedness_Q-H} we review, for the limit system \eqref{Q-H_E_intro}, the results presented in \cite{C-F_RWA} and \cite{C-F_sub}, and 
we explicitly derive the lifespan of solutions to equations \eqref{Q-H_E_intro} (see relation \eqref{improved_low_bound}).

In the last section, we deal with the lifespan analysis for system \eqref{full Euler} and we point out some consequences of the continuation criterion we have established (see in particular Subsection \ref{ss:cont_criterion+consequences}).

\paragraph*{Some notation and conventions.} \label{ss:notations}

The symbol $C_c^\infty (\R^2)$ denotes the space of $\infty$-times continuously differentiable functions on $\R^2$, having compact support in $\R^2$. The space $\mc D^{\prime}(\R^2)$ is the space of
distributions on $\R^2$. 
We use also the notation $C^0_w([0,T];X)$, with $X$ a Banach space, to refer to the space of continuous in time functions with values in $X$ endowed with its weak topology. 
Given $p\in[1,+\infty]$, by $L^p(\R^2)$ we mean the classical space of Lebesgue measurable functions $g$, where $|g|^p$ is integrable over the set $\R^2$ (with the usual modifications for the case $p=+\infty$).
We use also the notation $L_T^p(L^q)$ to indicate the space $L^p\big([0,T];L^q(\R^2)\big)$ with $T>0$.
Given $k \geq 0$, we denote by $H^{k}(\R^2)$ the Sobolev space of functions which belongs to $L^2(\R^2)$ together with all their derivatives up to order $k$. 
Moreover, the notation $B^s_{p,r}(\R^2)$ stands for the Besov spaces in $\R^2$ that are interpolation spaces between the Sobolev ones (we refer to Paragraph \ref{ss:LP_theory} in the Appendix for a more detailed discussion). 

For the sake of simplicity, we will omit from the notation the set $\R^2$, that we will explicitly point out if needed.

In the whole paper, the symbols $c $ and $C$ will denote generic multiplicative constants, which may change from line to line, and which do not depend on the small parameter $\veps$.
Sometimes, we will explicitly point out the quantities that these constants depend on, by putting them inside brackets. We agree to write $f\sim g$ whenever we have $c\, g\leq f \leq C\, g$.

Let $\big(f_\veps\big)_{0<\veps\leq1}$ be a family of functions in a normed space $X$. If this family is bounded in $X$,  we use the notation $\big(f_\veps\big)_{\veps} \subset X$.

\subsection*{Acknowledgements}
{\small
The author is member of the Italian Institute for Advanced Mathematics (INdAM) group and his work has been partially supported by the project CRISIS (ANR-20-CE40-0020-01), operated by the French National Research Agency (ANR). 

The author also acknowledge Daniele Del Santo and Francesco Fanelli for their insightful remarks. 

\section{Setting of the problem and main results} \label{s:result}

In this section, we formulate our working hypotheses (Subsection \ref{ss:FormProb}) and we state our main results
(Subsection \ref{ss:results}).

\subsection{Formulation of the problem} \label{ss:FormProb}

In this subsection, we present the rescaled density-dependent Euler equations with the Coriolis force, which we are going to consider in our study, and we formulate the main working hypotheses.

To begin with, let us introduce the ``primitive system'', that is the rescaled incompressible Euler system \eqref{full Euler}, supplemented with the scaling \eqref{scaling}, where $\veps \in \,]0,1]$ is a small parameter. Thus, the system consists of continuity equation (conservation of mass), the momentum equation and the divergence-free condition: respectively 
\begin{equation}\label{Euler_eps}
\begin{cases}
\d_t \vrho_{\veps} +\div (\vrho_{\veps} \vu_{\veps})=0\\
\d_t (\vrho_{\veps} \vu_{\veps})+\div (\vrho_{\veps} \vu_{\veps} \otimes \vu_{\veps})+ \frac{1}{\veps}\vrho_{\veps} \vu_{\veps}^{\perp}+\frac{1}{\veps}\nabla \Pi_\veps=0\\
\div \vu_{\veps} =0\, .
\end{cases}
\end{equation}
The unknowns are the fluid mass density $\vre=\vre(t,x)\geq0$ and its velocity field $\ue=\ue(t,x)\in\R^2$ with  $t\in \R_+$, $x\in \R^2$.

In \eqref{Euler_eps}, the pressure term has to scale like $1/\veps$, since it is the only force that allows to compensate the effect of fast rotation, at the geophysical scale.

From now on, in order to make condition \eqref{Lip_assumption} holds, we fix 
$$ s>2\, . $$

We assume that the initial density is a small perturbation of a constant profile.
Namely, we consider initial densities of the following form:
	\begin{equation}\label{in_vr}
	\vrez = 1 + \ep \, R_{0,\veps} 
	\end{equation}
where we suppose $R_{0,\veps}$ to be a bounded measurable function satisfying the controls
	\begin{align}
&\sup_{\veps\in\,]0,1]}\left\|  R_{0,\veps} \right\|_{L^\infty(\R^2)}\,\leq \, C\, ,\label{hyp:ill_data_R_0}\\
&\sup_{\veps\in\,]0,1]}\left\| \nabla R_{0,\veps} \right\|_{H^{s-1}(\R^2)}\,\leq \, C \label{hyp:ill_data_nablaR_0}
	\end{align}
and the initial mass density is bounded and bounded away from zero, i.e. for all $\veps \in\;]0,1]$: 
\begin{equation}\label{assumption_densities}
0<\underline{\vrho}\leq   \vrho_{0,\veps}(x) \,\leq \, \overline{\vrho}\, , \qquad x\in \R^2
\end{equation}
where $\underline{\vrho},\overline{\vrho}>0$ are positive constants.

As for the initial velocity fields, due to framework needed for the well-posedness issues, we require the following uniform bound
\begin{equation}\label{hyp:data_u_0}
\sup_{\veps\in\,]0,1]}\left\|  \vu_{0,\veps} \right\|_{H^s(\R^2)}\,\leq \, C\, .
\end{equation}
Thanks to the previous uniform estimates, we can assume (up to passing to subsequences) that there exist $R_0 \in W^{1,\infty}(\R^2)$, with $\nabla R_0\in H^{s-1}(\R^2)$, and $\vec u_0\in H^s(\R^2)$ such that 
\begin{equation}\label{init_limit_points}
\begin{split}
R_0:=\lim_{\veps \rightarrow 0}R_{0,\veps} \quad &\text{in}\quad L^\infty (\R^2) \\
\nabla R_0:=\lim_{\veps \rightarrow 0}\nabla R_{0,\veps}\quad &\text{in}\quad H^{s-1} (\R^2)\\
\vu_0:=\lim_{\veps \rightarrow 0}\vu_{0,\veps}\quad &\text{in}\quad H^s (\R^2)\, ,
\end{split}
\end{equation} 
where we agree that the previous limits are taken in the corresponding weak-$\ast$ topology.

\subsection{Main results}\label{ss:results}

\medbreak
We can now state our main results. We recall the notation $\big(f_\veps\big)_{\veps} \subset X$ to denote that the family $\big(f_\veps\big)_{\veps}$ is uniformly (in $\veps$) bounded in $X$.

The following theorem establishes the local well-posedness of system \eqref{Euler_eps} in the Sobolev spaces $B^s_{2,2}\equiv H^s$ (see Section \ref{s:well-posedness_original_problem}) and gives a lower bound for the lifespan of solutions (see Section \ref{s:lifespan_full}). 

\begin{theorem}\label{W-P_fullE}
For any $\veps \in\, ]0,1]$, let initial densities $\vrho_{0,\veps}$ be as in \eqref{in_vr} and satisfy the controls \eqref{hyp:ill_data_R_0} to \eqref{assumption_densities}. Let $\vu_{0,\veps}$ be divergence-free vector fields such that $\vu_{0,\veps} \in H^s(\R^2)$ for $s>2$. \\
Then, for any $\veps>0$, 
there exists a time $T_\veps^\ast >0$
such that 
the system \eqref{Euler_eps} has a unique solution $(\vrho_\veps, \vu_\veps, \nabla \Pi_\veps)$ where 
\begin{itemize}
\item $\vrho_\veps$ belongs to the space $C^0([0,T_\veps^\ast]\times \R^2)$ with $\nabla \vrho_\veps \in  C^0([0,T_\veps^\ast]; H^{s-1}(\R^2))$;
\item $\vu_\veps$ belongs to the space $C^0([0,T_\veps^\ast]; H^s(\R^2))$;
\item $\nabla \Pi_\veps$ belongs to the space $C^0([0,T_\veps^\ast]; H^s(\R^2))$.
\end{itemize}
Moreover, the lifespan $T_\veps^\ast$ of the solution to the two-dimensional density-dependent incompressible Euler equations with the Coriolis force is bounded from below by
\begin{equation}\label{improv_life_fullE}
\frac{C}{\|\vec u_{0,\veps}\|_{H^s}}\log\left(\log\left(\frac{C\, \|\vec u_{0,\veps}\|_{H^s}}{\max \{\mc A_\veps(0),\, \veps \, \mc A_\veps(0)\, \|\vec u_{0,\veps}\|_{H^s}\}}+1\right)+1\right)\, ,
\end{equation}
where $\mc A_\veps (0):= \|\nabla R_{0,\veps}\|_{H^{s-1}}+\veps\, \|\nabla R_{0,\veps}\|_{H^{s-1}}^{\lambda +1}$, for some suitable $\lambda\geq 1$.

In particular, there exists a time $T^\ast >0$ such that $$\inf_{\veps>0}T_\veps^\ast \geq T^\ast >0\, .$$
\end{theorem} 
Looking at \eqref{improv_life_fullE}, we stress the fact that the only fast rotational effects are not enough to state a global well-posedness result for system \eqref{Euler_eps}: this is coherent with the previous results, as the one in \cite{D-F_JMPA}. 

Now, once we have stated the local in time well-posedness for system \eqref{Euler_eps} in the Sobolev spaces $H^s$, in Section \ref{s:sing-pert} we address the singular perturbation problem describing, in a rigorous way, the limit dynamics depicted by the quasi-homogeneous incompressible Euler system \eqref{system_Q-H_thm} below. 
\begin{theorem}\label{thm:limit_dynamics}
Let $s>2$. For any fixed value of $\veps \in \; ]0,1]$, let initial data $\left(\vrho_{0,\veps},\vec u_{0,\veps}\right)$ verify the hypotheses fixed in Paragraph \ref{ss:FormProb}, and let
$\left( \vre, \ue\right)$ be a corresponding solution to system \eqref{Euler_eps}.
Let $\left(R_0,\vec u_0\right)$ be defined as in \eqref{init_limit_points}.

Then, one has the following convergence properties:
	\begin{align*}
	\varrho_\ep \rightarrow 1 \qquad\qquad \mbox{ in } \qquad &L^\infty\big([0,T^\ast]; L^\infty(\R^2 )\big)\,, \\
	R_\veps:=\frac{\varrho_\ep - 1}{\ep}  \weakstar R \qquad\qquad \mbox{ in }\qquad &L^\infty\bigl([0,T^\ast]; L^\infty(\R^2)\bigr)\,, \\
	\nabla R_\veps  \weakstar \nabla R \qquad\qquad \mbox{ in }\qquad &L^\infty\bigl([0,T^\ast]; H^{s-1}(\R^2)\bigr)\,, \\
 \vec{u}_\ep \weakstar \vec{u}
	\qquad\qquad \mbox{ in }\qquad &L^\infty\big([0,T^\ast];H^s(\R^2)\big)\, .
	\end{align*}	
In addition, $\Big(R\, ,\, \vec{u}  \Big)$ is a solution
to the quasi-homogeneous incompressible Euler system  in $\R_+ \times \R^2$:
\begin{equation}\label{system_Q-H_thm}
\begin{cases}
\d_t R+\div (R\vu)=0 \\
\d_t \vec u+\div \left(\vec{u}\otimes\vec{u}\right)+R\vu^\perp+ \nabla \Pi =0  \\
\div \vec u\,=\,0
\end{cases}
\end{equation}
where $\nabla \Pi$ is a suitable pressure term belonging to $L^\infty\big([0,T^\ast];H^s(\R^2)\big)$. 
\end{theorem}

\begin{remark} 
Due to the fact that the system \eqref{system_Q-H_thm} is well-posed in the previous functional setting (see Theorem \ref{thm:well-posedness_Q-H-Euler} below), we get the convergence of the whole sequence of weak solutions to the solutions of the target equations on the large time interval where the weak solutions to the primitive equations exist.
\end{remark}

At the limit, we have found that the dynamics is prescribed by the quasi-homogeneous incompressible Euler system \eqref{system_Q-H_thm}, for which we state the local well-posedness in $H^s$ (see Section \ref{s:well-posedness_Q-H}). It is worth to remark that the global well-posedness issue for this system is still an open problem. 

\begin{theorem}\label{thm:well-posedness_Q-H-Euler}
Take $s>2$. Let $\big(R_0,u_0 \big)$ be initial data such that $R_0\in L^{\infty}(\R^2)$ and
$\vu_0 \,\in  H^s(\R^2)$, with $\nabla R_0\in H^{s-1}(\R^2)$ and $\div \vu_0\,=\,0$.

Then, there exists a time $T^\ast > 0$ such that, on $[0,T^\ast]\times\R^2$, problem \eqref{system_Q-H_thm} has a unique solution $(R,\vu, \nabla \Pi)$ with the following properties:
\begin{itemize}
 \item $R\in C^0\big([0,T^\ast]\times \R^2\big)$ and $\nabla R\in C^0\big([0,T^\ast];H^{s-1}(\R^2)\big)$;
 \item $\vu$ belongs to $C^0\big([0,T^\ast]; H^{s}(\R^2)\big)$;
 \item the pressure term $\nabla \Pi$ is in $C^0\big([0,T^\ast];H^s(\R^2)\big)$. 
 \end{itemize}

In addition, the lifespan $T^\ast>0$ of the solution $(R, \vu, \nabla \Pi)$ to the $2$-D quasi-homogeneous Euler system \eqref{system_Q-H_thm} enjoys the following lower bound:
\begin{equation}\label{improved_low_bound}
T^\ast\geq \frac{C}{\|\vu_0\|_{H^s}}\log\left(\log \left(C\frac{\|\vu_0\|_{H^s}}{\|R_0\|_{L^\infty}+\|\nabla R_0\|_{H^{s-1}}}+1\right)+1\right)\, ,
\end{equation}
where $C>0$ is a ``universal'' constant, independent of the initial datum.

\end{theorem}
The proof of the previous ``asymptotically global'' well-posedness result is presented in Subsection \ref{ss:improved_lifespan}. 

\section{Well-posedness for the original problem}\label{s:well-posedness_original_problem}
This section is devoted to the well-posedness issue in the $H^s$ spaces stated in Theorem \ref{W-P_fullE}. We recall that, due to the Littlewood-Paley theory, we have the equivalence between $H^s$ and $B^s_{2,2}$ spaces (see Appendix for details).

We also underline that in this section we keep $\veps \in \; ]0,1] $ fixed. We will take care of explicitly pointing out the dependence to the Rossby number in all the computations in order to get controls that are uniform with respect to the $\veps$-parameter. The choice in keeping explicit the dependence on the rotational parameter is motivated by the fact that we will perform the fast rotation limit (see Section \ref{s:sing-pert} below). 

First of all, since $\vrho_\veps$ is a small perturbation of a constant profile, we set 
\begin{equation}\label{def_a_veps}
 \alpha_\veps :=\frac{1}{\vrho_\veps}-1=\veps a_\veps\quad \text{with}\quad a_\veps:=-R_\veps/\vrho_\veps \, .
\end{equation}
The choice of looking at $\alpha_\veps$ is dictated by the fact that we will extensively exploit the elliptic equation \eqref{elliptic_eq}. 

Now, using the divergence-free condition, we can rewrite the system \eqref{Euler_eps} in the following way (see also Lemma 3 in \cite{D-F_JMPA}): 
\begin{equation}\label{Euler-a_eps_1}
\begin{cases}
\d_t a_{\veps} +\vu_\veps \cdot \nabla a_\veps=0\\
\d_t \vu_{\veps}+ \vu_{\veps} \cdot \nabla \vu_{\veps}+ \frac{1}{\veps} \vu_{\veps}^{\perp}+(1+\veps a_\veps)\frac{1}{\veps}\nabla \Pi_\veps=0\\
\div \vu_{\veps} =0
\end{cases}
\end{equation}
with the initial condition $(a_\veps, \ue)_{|t=0}=(a_{0,\veps},\vu_{0,\veps})$. 

We start by presenting the proof of existence of solutions at the claimed regularity. 
For that scope, we follow a standard procedure: first, we construct a sequence of smooth approximate solutions. Next, we deduce uniform bounds (with respect to the approximation parameter and also to $\veps$) for those regular solutions. Finally, by use of those uniform bounds and an energy method, together with an interpolation argument, we are able to take the limit in the approximation parameter and gather the existence of a solution to the original problem.  

We end this Section \ref{s:well-posedness_original_problem}, proving uniqueness of solutions in the claimed functional setting, by using a relative entropy method.

\subsection{Construction of smooth approximate solutions}
For any $n\in \N$, let us define 
\begin{equation*}
(a_{0,\veps}^n, \vec u_{0,\veps}^n ):= (S_n a_{0,\veps}, S_n \vec u_{0,\veps})\, ,
\end{equation*}
where $S_{n}$ is the low frequency cut-off operator introduced in \eqref{eq:S_j} in the Appendix. We stress also the fact that $a_{0,\veps}\in C^0_{\rm loc}$, since $a_{0,\veps}$ and $\nabla a_{0,\veps}$ are in $L^\infty$.

Then, for any $n\in \N$, we have the density functions $a_{0,\veps}^n\in L^\infty$. Moreover, one has that $\nabla a_{0,\veps}^n$ and $\vec u_{0,\veps}^n$ belong to $H^\infty:=\bigcap_{\sigma \in \R} H^\sigma$ which is embedded (for a suitable topology on $H^\infty$) in the space $C_b^\infty$ of $C^\infty$ functions which are globally bounded together with all their derivatives. 

In addition, by standard properties of mollifiers, one has the following strong convergences 
\begin{equation}\label{conv_in_data}
\begin{split}
a^n_{0,\veps}\rightarrow a_{0,\veps}\quad  &\text{in}\quad C_{\rm loc}^{0} \\
\nabla a^n_{0,\veps}\rightarrow \nabla a_{0,\veps}\quad  &\text{in}\quad  H^{s-1}\\
\vu_{0,\veps}^n\rightarrow \vu_{0,\veps} \quad &\text{in}\quad H^s \, .
\end{split}
\end{equation}

This having been established, we are going to define a sequence of approximate solutions to system \eqref{Euler-a_eps_1} by induction. First of all, we set $(a_\veps^0,\vu_\veps^0, \nabla \Pi_\veps^0)=(a^0_{0,\veps},\vu^0_{0,\veps},0)$.  Then, for all $\sigma \in \R$, we have that $\nabla a_{\veps}^0 ,\vec u_{\veps}^0 \in H^\sigma$ and $a^0_\veps \in L^\infty$ with $\div \vec u_{\veps}^0=0$. Next, assume that the couple $(a_\veps^n, \ue^n)$ is given such that, for all $\sigma \in \R$,
\begin{equation*}
 a^n_\veps \in C^0(\R_+;L^\infty) \quad \nabla a_{\veps}^n ,\vec u_{\veps}^n \in C^0(\R_+; H^\sigma)\quad  \quad \text{and}\quad \div \vec u_{\veps}^n=0\, .
\end{equation*}
First of all, we define $a_\veps^{n+1}$ as the unique solution to the linear transport equation
\begin{equation}\label{mass_eq}
\d_t a_{\veps}^{n+1} +\vu_\veps^n \cdot \nabla a_\veps^{n+1}=0 \quad \text{with}\quad {(\an)}_{|t=0}=a_{0,\veps}^{n+1}\, .
\end{equation}
Since, by inductive hypothesis and embeddings, $\ue^n$ is divergence-free, smooth and uniformly bounded with all its derivatives, we can deduce that $\an \in L^\infty(\R_+; L^\infty )$. Moreover, from

\begin{equation*}
\d_t\, \d_i \an +\ue^{n}\cdot \nabla\, \d_i \an =-\d_i \ue^n \cdot \nabla \an \quad \text{with}\quad {(\d_i\an)}_{|t=0}=\d_i a_{0,\veps}^{n+1} \quad \text{for}\; i=1,2
\end{equation*} 
and thanks to the Theorem \ref{thm_transport}, we can propagate all the $H^{\sigma}$ norms of the initial datum. We deduce that $a_\veps^{n+1}\in C^0(\R_+;L^\infty)$ and $\nabla a_\veps^{n+1}\in C^0(\R_+;H^{\sigma})$ for any $\sigma \in \R$. 
Next, we consider the approximate linear iteration
\begin{equation}\label{approx_iteration_1}
\begin{cases}
\d_t \vu_{\veps}^{n+1}+ \vu_{\veps}^n \cdot \nabla \vu_{\veps}^{n+1}+ \frac{1}{\veps} \vu_{\veps}^{\perp, n+1}+(1+\veps a_\veps^{n+1})\frac{1}{\veps}\nabla \Pi_\veps^{n+1}=0\\
\div \vu_{\veps}^{n+1} =0\\
(\vu_\veps^{n+1})_{|t=0}=\vu_{0,\veps}^{n+1}\, .
\end{cases}
\end{equation}
At this point, one can solve the previous linear problem finding a unique solution $\vu_\veps^{n+1}\in C^0(\R_+; H^\sigma )$ for any $\sigma \in \R$ and the pressure term $\nabla \Pi_\veps^{n+1}$ can be uniquely determined (we refer to \cite{D_AT} for details in this respect).

\subsection{Uniform estimates for the approximate solutions}\label{ss:unif_est}

We now have to show (by induction) uniform bounds for the sequence $(a_\veps^n, \ue^n, \nabla \Pi_\veps^n)_{n\in \N}$ we have constructed above.

We start by finding uniform estimates for $a_\veps^{n+1}$. Thanks to equation \eqref{mass_eq} and the divergence-free condition on $\ue^n$, we can propagate the $L^\infty$ norm for any $t\geq 0$:
\begin{equation}\label{eq:transport_density}
\|a_\veps^{n+1}(t)\|_{L^\infty} \leq\|a_{0,\veps}^{n+1}\|_{L^\infty}\leq C \|a_{0,\veps}\|_{L^\infty}\, .
\end{equation} 
At this point we want to estimate $\nabla a_\veps^{n+1}$ in $H^{s-1}$. We have for $i=1,2$:
\begin{equation*}
\d_t\, \d_i \an +\ue^{n}\cdot \nabla\, \d_i \an =-\d_i \ue^n \cdot \nabla \an \, .
\end{equation*}
Taking the non-homogeneous dyadic blocks $ \Delta_j$, we obtain 
\begin{equation*}
\d_t \Delta_j \, \d_i a_\veps^{n+1}+\ue^n\cdot \nabla  \Delta_j\, \d_i \an=[\ue^n\cdot \nabla,\Delta_j]\, \d_i \an -\Delta_j\left(\d_i \ue^n \cdot \nabla \an \right)\, .
\end{equation*}
Multiplying by $ \Delta_j\, \d_i \an$, we have
\begin{equation*}
\| \Delta_j\, \d_i \an(t)\|_{L^2}\leq  \| \Delta_j\, \d_i a_{0,\veps}^{n+1}\|_{L^2}+C \int_0^t\left(\left\|[\ue^n\cdot \nabla, \Delta_j] \, \d_i \an\right\|_{L^2}+\|\Delta_j\left(\d_i \ue^n \cdot \nabla \an \right) \|_{L^2} \right)\, \detau  \, .
\end{equation*}
We apply now the second commutator estimate stated in Lemma \ref{l:commutator_est} to the former term in the integral on the right-hand side, getting
\begin{equation*}
2^{j(s-1)}\left\|[\ue^n\cdot \nabla, \Delta_j]\, \d_i \an\right\|_{L^2}\leq C\, c_j(t)\left(\|\nabla \ue^n\|_{L^\infty}\|\d_i \an\|_{ H^{s-1}}+\|\nabla \ue^n\|_{H^{s-1}}\|\d_i \an \|_{L^\infty}\right)
\end{equation*}
where $(c_j(t))_{j\geq -1}$ is a sequence in the unit ball of $\ell^2$.

Instead, the latter term can be bounded in the following way:
\begin{equation*}\label{eq:nabla-u_nabla-a}
2^{j(s-1)}\|\Delta_j \left(\d_i \ue^n \cdot \nabla \an \right)\|_{L^2}\leq C\, c_j(t)\,  \|\nabla \ue^n\|_{H^{s-1}}\|\nabla \an \|_{H^{s-1}}\, .
\end{equation*}

Then, due to the embedding $H^\sigma(\R^2)\hookrightarrow L^\infty (\R^2)$ for $\sigma>1$,
\begin{equation*}
2^{j(s-1)}\| \Delta_j\, \nabla \an(t)\|_{L^2}\leq 2^{j(s-1)}\| \Delta_j \nabla a_{0,\veps}^{n+1}\|_{L^2}+\int_0^t C\, c_j(\tau )\left(\| \ue^n\|_{ H^s}\|\nabla \an\|_{ H^{s-1}}\right)\, \detau\, . 
\end{equation*}
At this point, after summing on indices $j\geq -1$, thanks to the Minkowski inequality (for which we refer to Proposition 1.3 of \cite{B-C-D}) combined with a Gr\"onwall type argument, we finally obtain 
\begin{equation}\label{est:a^(n+1)}
\sup_{0\leq t\leq T}\|\nabla \an (t)\|_{H^{s-1}}\leq \|\nabla a_{0,\veps}^{n+1}\|_{H^{s-1}}\, \exp \left(\int_0^T C \, \|\ue^n(t)\|_{ H^s}\, \dt\right)\, .
\end{equation}

Now, we have to estimate the velocity field $\ue^{n+1}$ and for that purpose we start with the $L^2$ estimate. We take the momentum equation in the original form:
\begin{equation}\label{mom_eq_original_prob}
\vrho_\veps^{n+1}\left(\d_t\ue^{n+1}+\ue^n\cdot \nabla \ue^{n+1}\right)+\frac{1}{\veps}\vrho_\veps^{n+1}\ue^{\perp, n+1}+\frac{1}{\veps}\nabla \Pi_\veps^{n+1}=0 \, ,
\end{equation}
where we construct $\vrho_\veps^{n+1}:=1/(1+\veps \an)$ starting from $\an$. Notice that $\vrho_\veps^{n+1}$ satisfies the transport equation
\begin{equation*}
\d_t \vrho_\veps^{n+1}+\ue^{n}\cdot \nabla \vrho_\veps^{n+1}=0\, .
\end{equation*}  

At this point, we test equation \eqref{mom_eq_original_prob} against $\ue^{n+1}$. We integrate by parts on $\R^2$, deriving the following estimate: 
$$ \int_{\R^2} \vrho_\veps^{n+1}\d_t|\ue^{n+1}|^2+\int_{\R^2}\vrho_\veps^{n+1}\ue^n\cdot \nabla |\ue^{n+1}|^2=0 \, ,$$
which implies (making use of the transport equation for $\vrho_\veps^{n+1}$)
\begin{equation*}
 \left\|\sqrt{\vrho_\veps^{n+1}(t)}\,  \ue^{n+1}(t)\right\|_{L^2}\leq \left\|\sqrt{\vrho_{0,\veps}^{n+1}} \, \vec u_{0,\veps}^{n+1}\right\|_{L^2}\, .
\end{equation*}
From the previous bound, due to the assumption \eqref{assumption_densities}, we  can deduce the preservation of the $L^2$ norm for the velocity field $\ue^{n+1}$:
\begin{equation*}
 \left\| \ue^{n+1}(t)\right\|_{L^2}\leq C\left\| \vec u_{0,\veps}^{n+1}\right\|_{L^2}\leq C\left\| \vec u_{0,\veps}\right\|_{L^2}\, .
\end{equation*}

Taking now the operator $ \Delta_j$ in the momentum equation in \eqref{approx_iteration_1}, we obtain
\begin{equation*}
\d_t \Delta_j \ue^{n+1}+\ue^n\cdot \nabla \Delta_j \ue^{n+1}=[\ue^n\cdot \nabla, \Delta_j]\ue^{n+1}-\frac{1}{\veps}\Delta_j\ue^{\perp ,n+1}-\Delta_j\left[\left(1+\veps a_\veps^{n+1}\right)\frac{1}{\veps}\nabla \Pi_\veps^{n+1}\right]
\end{equation*}
and multiplying again by $\Delta_j\ue^{n+1}$, we have cancellations so that 
\begin{equation}\label{eq:vel_est_dyadic}
\left\|\Delta_j\ue^{n+1}(t)\right\|_{L^2}\leq \left\|\Delta_j\vec u_{0,\veps}^{n+1}\right\|_{L^2}+C\int_0^t \left(\left\|[\ue^n\cdot \nabla, \Delta_j]\ue^{n+1}\right\|_{L^2}+\left\|\Delta_j\left( a_\veps^{n+1}\nabla \Pi_\veps^{n+1}\right)\right\|_{L^2}\right) \, \detau\, .
\end{equation}
As done before, we employ here the commutator estimates of Lemma \ref{l:commutator_est} in order to have 
\begin{equation*}
\begin{split}
2^{js} \left\|[\ue^n\cdot \nabla,  \Delta_j]\ue^{n+1}\right\|_{L^2}&\leq C\,c_j\, \left(\|\nabla \ue^n\|_{L^\infty}\|\ue^{n+1}\|_{ H^s}+\|\nabla \ue^{n+1}\|_{L^\infty}\|\ue^n\|_{ H^s}\right)\\
&\leq C\,c_j\, \left(\| \ue^n\|_{ H^s}\|\ue^{n+1}\|_{ H^s}\right)\, .
\end{split}
\end{equation*}
For the latter term on the right-hand side of \eqref{eq:vel_est_dyadic}, we take advantage of the Bony decomposition (see Paragraph \ref{app_paradiff}) and apply Proposition \ref{prop:app_fine_tame_est}. We may infer that 
\begin{equation*}
\begin{split}
 \left\|a_\veps^{n+1}\nabla \Pi_\veps^{n+1}\right\|_{ H^s}\leq  C\left(\|\an \|_{L^\infty}+\|\nabla a_\veps^{n+1}\|_{ H^{s-1}}\right)\|\nabla \Pi_\veps^{n+1}\|_{ H^s}\, .
\end{split}
\end{equation*}
To finish with, we have to find a uniform bound for the pressure term. For that, we apply the $\div$ operator in \eqref{approx_iteration_1}. Thus, we aim at solving the elliptic problem 
\begin{equation}\label{eq:elliptic_problem}
-\div \left(\left(1+\veps a_\veps^{n+1}\right)\nabla \Pi_\veps^{n+1}\right)=\, \veps\, \div (\ue^n \cdot \nabla \ue^{n+1} )- \curl \ue^{n+1}\, .
\end{equation}
Thanks to the assumption \eqref{assumption_densities} and Lemma 2 of \cite{D_JDE}, we can obtain 
\begin{equation}\label{est:Pi_L^2}
\begin{split}
\|\nabla \Pi_\veps^{n+1}\|_{L^2}&\leq C\left(\veps \, \|\ue^n\cdot \nabla \ue^{n+1}\|_{L^2}+\|\ue^{\perp,n+1}\|_{L^2}\right)\\
&\leq  C\left(\veps \, \|\ue^n\|_{L^2} \| \ue^{n+1}\|_{H^s}+\|\ue^{n+1}\|_{L^2}\right)\, .
\end{split}
\end{equation} 
Now, we apply the spectral cut-off operator $ \Delta_j$ to \eqref{eq:elliptic_problem}. We get 
\begin{equation*}
-\, \div \left( A_\veps \Delta_j\nabla \Pi_\veps^{n+1}\right)=\div \left(\left[ \Delta_j,A_\veps\right]\nabla \Pi_\veps^{n+1}\right)+\, \div  \Delta_j F_\veps
\end{equation*}
for all $j\geq 0$ and where we have defined $A_\veps:=\left(1+\veps a_\veps^{n+1}\right)$ and $F_\veps:=\veps \ue^n \cdot \nabla \ue^{n+1}+\ue^{\perp,n+1}$.
Hence multiplying both sides by $ \Delta_j  \Pi_\veps^{n+1}$ and integrating over $\R^2$, we have 
\begin{equation*}
-\int_{\R^2}\Delta_j\Pi_\veps^{n+1} \div \left( A_\veps \Delta_j\nabla \Pi_\veps^{n+1}\right) \, \dx= \int_{\R^2}\Delta_j\Pi_\veps^{n+1} \div \left( \left[\Delta_j, A_\veps\right] \nabla \Pi_\veps^{n+1}\right) \, \dx+\int_{\R^2}\Delta_j\Pi_\veps^{n+1} \div \Delta_j F_\veps \, \dx.
\end{equation*}
Since for $j\geq 0$ we have $\|\Delta_j \nabla \Pi_\veps^{n+1}\|_{L^2}\sim 2^j\|\Delta_j  \Pi_\veps^{n+1}\|_{L^2}$ (according to Lemma \ref{l:bern}) and using H\"older's inequality for the right-hand side, we obtain for all $j\geq 0$: 
\begin{equation*}
2^{j}\| \Delta_j \nabla \Pi_\veps^{n+1}\|^2_{L^2}\leq C\| \Delta_j \nabla \Pi_\veps^{n+1}\|_{L^2}\left( \| \div \left[ \Delta_j,A_\veps \right]\nabla \Pi_\veps^{n+1}\|_{L^2}+\|\div \Delta_j F_\veps\|_{L^2} \right)\, .
\end{equation*}
To deal with the former term on the right-hand side, we take advantage of the following commutator estimate (see Lemma \ref{l:commutator_pressure} in the Appendix):
\begin{equation*}
\| \div \left[ \Delta_j,A_\veps \right]\nabla \Pi_\veps^{n+1}\|_{L^2}\leq C\, c_j \, 2^{-j(s-1)}\|\nabla A_\veps\|_{H^{s-1}}\|\nabla \Pi_\veps^{n+1}\|_{H^{s-1}}
\end{equation*}
for a suitable sequence $(c_j)_{j}$ belonging to the unit sphere of $\ell^2$.

After multiplying by $2^{j(s-1)}$, we get
 \begin{equation*}
2^{js}\| \Delta_j \nabla \Pi_\veps^{n+1}\|_{L^2}\leq C\left(c_j\,  \|\nabla A_\veps\|_{H^{s-1}}\|\nabla \Pi_\veps^{n+1}\|_{H^{s-1}}+2^{j(s-1)}\|\div \Delta_j F_\veps\|_{L^2} \right)\, .
\end{equation*}
Taking the $\ell^2$ norm of both sides and summing up the low frequency blocks related to $\Delta_{-1}\nabla \Pi_\veps^{n+1}$, we may have 
\begin{equation*}
\|\nabla \Pi_\veps^{n+1}\|_{H^s}\leq C\left(  \|\nabla A_\veps\|_{H^{s-1}}\|\nabla \Pi_\veps^{n+1}\|_{H^{s-1}}+\|\div  F_\veps\|_{H^{s-1}}+\|\Delta_{-1}\nabla \Pi_\veps^{n+1}\|_{L^2} \right)\, .
\end{equation*}
We observe that $\|\Delta_{-1}\nabla \Pi_\veps^{n+1}\|_{L^2}\leq C\|\nabla \Pi_\veps^{n+1}\|_{L^2}$ and 
\begin{equation*}
\|\nabla \Pi_\veps^{n+1}\|_{H^{s-1}}\leq C\|\nabla \Pi_\veps^{n+1}\|_{L^2}^{1/s}\|\nabla \Pi_\veps^{n+1}\|_{H^s}^{1-1/s}\, .
\end{equation*}
Therefore,
\begin{equation*}
\|\nabla \Pi_\veps^{n+1}\|_{H^s}\leq C\left(  \|\nabla A_\veps\|_{H^{s-1}}\|\nabla \Pi_\veps^{n+1}\|_{L^2}^{1/s}\|\nabla \Pi_\veps^{n+1}\|_{H^{s}}^{1-1/s}+\|\div  F_\veps\|_{H^{s-1}}+\|\nabla \Pi_\veps^{n+1}\|_{L^2} \right)\, .
\end{equation*}
Then applying Young's inequality we finally infer that 
\begin{equation}\label{est_Pi_H^s_1}
\|\nabla \Pi_\veps^{n+1}\|_{H^s}\leq C\left(  \left(1+\|\nabla A_\veps\|_{H^{s-1}}\right)^s \|\nabla \Pi_\veps^{n+1}\|_{L^2}+\|\div  F_\veps\|_{H^{s-1}} \right)\, .
\end{equation}
It remains to analyse the term $\div F_\veps$ where $F_\veps:=\veps \ue^n \cdot \nabla \ue^{n+1}+\ue^{\perp,n+1}$. Due to the divergence-free conditions, we can write 
\begin{equation*}
\div (\ue^n \cdot \nabla \ue^{n+1} )=\nabla \ue^{n}:\nabla \ue^{n+1}
\end{equation*}
and as $H^{s-1}$ is an algebra, the term $\div (\ue^n \cdot \nabla \ue^{n+1} )$ is in $H^{s-1}$, with 
\begin{equation}\label{est:div_u}
\|\div (\ue^n \cdot \nabla \ue^{n+1} )\|_{H^{s-1}}\leq C \|\ue^n\|_{H^s}\|\ue^{n+1}\|_{H^s}\, .
\end{equation}
Putting \eqref{est:Pi_L^2} and \eqref{est:div_u} into \eqref{est_Pi_H^s_1}, we find that 
\begin{equation}\label{est_Pi_final}
\begin{split}
\|\nabla \Pi_\veps^{n+1}\|_{ H^{s}}&\leq C  \left(1+\veps \|\nabla a_\veps^{n+1}\|_{ H^{s-1}}\right)^s \left(\veps \|\ue^n\|_{L^2}\|\ue^{n+1}\|_{ H^s}+\|\ue^{\perp,n+1}\|_{L^2}\right)\\
&+C\, \left(\veps \|\ue^{n}\|_{ H^{s}}\|\ue^{n+1}\|_{ H^{s}} +\|\ue^{\perp,n+1}\|_{ H^{s}}\right)\\
&\leq C  \left(1+\veps \|\nabla a_\veps^{n+1}\|_{ H^{s-1}}\right)^s \left(\veps \|\ue^n\|_{H^s}+1\right)\|\ue^{n+1}\|_{ H^s}\, ,
\end{split}
\end{equation} 
which implies the $L^\infty_T(H^s)$ estimate for the pressure term:
\begin{equation}\label{est:Pi^(n+1)}
\|\nabla \Pi_\veps^{n+1}\|_{L^{\infty}_T H^{s}}\leq C \left(1+\veps \|\nabla a_\veps^{n+1}\|_{L^\infty_T  H^{s-1}}\right)^s\left(\veps \|\ue^n\|_{L^\infty_T H^s}+1\right)\|\ue^{n+1}\|_{L^\infty_TH^s}\, .
\end{equation}
Combining all the previous estimates together with a Gr\"onwall type inequality, we finally obtain an estimate for the velocity field:
\begin{equation}\label{est:u^(n+1)}
\sup_{0\leq t\leq T}\|\ue^{n+1}(t)\|_{H^s}\leq \|\vec u_{0,\veps}^{n+1}\|_{H^s}\exp \left(\int_0^T A_n(t)\, \dt\right)
\end{equation}
where 
\begin{equation*}
\begin{split}
A_n(t)=C\left(\|\an (t)\|_{L^\infty}+\|\nabla a_\veps^{n+1}(t)\|_{ H^{s-1}}\right)\left(1+\veps \|\nabla a_\veps^{n+1}(t)\|_{ H^{s-1}}\right)^s\left(\veps \|\ue^n(t)\|_{H^{s}}+1\right)+C\|\ue^n(t)\|_{ H^s}\, .
\end{split}
\end{equation*}

We point out that the above constants $C$ do not depend on $n$ nor on $\veps$.

The scope in what follows is to obtain uniform estimates by induction. Thanks to the assumptions stated in Paragraph \ref{ss:FormProb}, we can suppose that the initial data satisfy 
\begin{equation*}
\|a_{0,\veps}\|_{L^\infty}\leq \frac{C_0}{2}\, ,\quad \quad \|\nabla a_{0,\veps}\|_{H^{s-1}}\leq \frac{C_1}{2} \quad \quad \text{and}\quad \quad \|\vec{u}_{0,\veps}\|_{H^s}\leq \frac{C_2}{2} \, ,
\end{equation*}
for some $C_0, C_1, C_2>0$. Due to the relation \eqref{eq:transport_density} we immediately infer that, for all $n\geq 0$, 
\begin{equation*}
\|\an \|_{L^\infty_t L^\infty}\leq C\|a_{0,\veps}\|_{L^\infty}\leq C\, C_0 \quad \text{for all }t\in \R_+.
\end{equation*}
At this point, the aim is to show (by induction) that the following uniform bounds hold  for all $n\geq 0$:
\begin{equation}\label{eq:unifbounds}
\begin{split}
&\|\nabla a_\veps^{n+1}\|_{L^\infty_{T^\ast}H^{s-1}}\leq C_1\\
&\|\ue^{n+1}\|_{L^\infty_{T^\ast}H^s}\leq C_2\\
&\|\nabla \Pi_\veps^{n+1}\|_{L^\infty_{T^\ast}H^s}\leq C_3\, ,
\end{split}
\end{equation}
provided that $T^\ast$ is sufficiently small.

The previous estimates in \eqref{eq:unifbounds} obviously hold for $n=0$. At this point, we will prove them for $n+1$, supposing that the controls in \eqref{eq:unifbounds} are true for $n$.
From \eqref{est:a^(n+1)}, \eqref{est:u^(n+1)} and \eqref{est:Pi^(n+1)} we obtain 
\begin{align*}
&\|\nabla a_\veps^{n+1}\|_{L^\infty_{T}H^{s-1}}\leq \frac{C_1}{2}\exp \Big(CTC_2\Big)\\
&\|\ue^{n+1}\|_{L^\infty_{T}H^s}\leq \frac{C_2}{2}\exp \Big(CT(C_0+C_1)\left(1+\veps C_1\right)^s(\veps C_2+1)C_2 \Big)\\
&\|\nabla \Pi_\veps^{n+1}\|_{L^\infty_{T}H^s}\leq C(\veps C_2+1)\left(1+\veps C_1\right)^s\|\ue^{n+1}\|_{L^\infty_{T}H^s}\, .
\end{align*}
So we can choose $T^{\ast}$ such that $\exp \Big(\max\{C_0+C_1,\, 1\}\, CT\left(1+ C_1\right)^s(1+ C_2)\, C_2 \Big)\leq 2$. Notice that $T^\ast$ does not depend on $\veps$. 
 
Thus, by induction, \eqref{eq:unifbounds} holds for the step $n+1$, and therefore it is true for any $n\in \N$. 

\subsection{Convergence}\label{ss:conv_H^s}
To prove the convergence, we estimate the difference between two iterations. First of all, let us define 
\begin{equation*}
\widetilde{a}_\veps^{n}:=a_\veps^n -a^n_{0,\veps}
\end{equation*}
that satisfies the transport equation
\begin{equation*}
\begin{cases}
\d_t \widetilde{a}^n_\veps+\ue^{n-1}\cdot \nabla \widetilde{a}_\veps^n=-\ue^{n-1}\cdot \nabla a_{0,\veps}^n\\
\widetilde{a}{_{\veps}^n}_{|t=0}=0\, .
\end{cases}
\end{equation*}
Hence, since the right-hand side is definitely uniformly bounded (with respect to $n$) in $L^1_{\rm loc}(\R_+;L^2)$, from classical results for transport equations we get that $(\widetilde{a}^n_\veps)_{n\in \N}$ is uniformly bounded in $C^0([0,T];L^2)$. Now, we want to prove that the sequence $(\widetilde{a}^n_\veps, \ue^n, \nabla \Pi^n_\veps)_{n\in \N}$ is a Cauchy sequence in $C^0([0,T];L^2)$. 
So, let us define, for $(n,l)\in \N^2$, the following quantities,
\begin{align*}
&\delta {a}_\veps^{n,l}:=\anl-a_\veps^n\\
&\delta\widetilde{a}_\veps^{n,l} :=\widetilde{a}_\veps^{n+l}-\widetilde{a}_\veps^{n}=\delta a_\veps^{n,l}-\delta a_{0,\veps}^{n,l}\, , \quad \text{ where }\quad \delta {a}_{0,\veps}^{n,l}:=a_{0,\veps}^{n+l}-a_{0,\veps}^n\\
&\delta \ue^{n,l}:=\ue^{n+l}-\ue^n\\
&\delta \Pi_\veps^{n,l}:=\Pi_\veps^{n+l}-\Pi_\veps^n\, ,
\end{align*}
that solve the following system
\begin{equation}\label{syst_convergence_analysis}
\begin{cases}
\d_t \delta \widetilde{a}_{\veps}^{n,l} +\vu_\veps^{n+l-1} \cdot \nabla \delta \widetilde{a}_\veps^{n,l}=-\delta\ue^{n-1,l}\cdot \nabla a_\veps^n-\ue^{n+l-1}\cdot \nabla \delta a^{n,l}_{0,\veps}\\
\d_t\delta \vu_{\veps}^{n,l}+ \vu_{\veps}^{n+l-1} \cdot \nabla \delta\vu_{\veps}^{n,l}=-\delta \ue^{n-1,l} \cdot \nabla \ue^n - \frac{1}{\veps} (\delta \vu_{\veps}^{n,l})^\perp-(1+\veps a_\veps^{n+l})\frac{1}{\veps}\nabla \delta \Pi_\veps^{n,l}-\delta a_\veps^{n,l} \nabla \Pi_\veps^n \\
\div  \delta \vu_{\veps}^{n,l} =0\\
(\delta \widetilde{a}_\veps^{n,l},\delta \vu_\veps^{n,l})_{|t=0}=(0,\delta \vu_{0,\veps}^{n,l})\, .
\end{cases}
\end{equation}

We perform an energy estimate for the first equation in \eqref{syst_convergence_analysis}, getting
\begin{equation*}
\|\delta \widetilde{a}_\veps^{n,l}(t)\|_{L^2}\leq C\int_0^t\left(\|\nabla a_\veps^n\|_{L^\infty}\|\delta \vu_\veps^{n-1,l}\|_{L^2}+\|\ue^{n+l-1}\|_{L^2}\|\nabla \delta a_{0,\veps}^{n,l}\|_{L^\infty}\right)\, \detau\, .
\end{equation*} 
Moreover, from the momentum equation multiplied by $\delta \ue^{n,l}$, integrating by parts over $\R^2$, we obtain 
\begin{equation*}
\begin{split}
\int_{\R^2}\frac{1}{2}\d_t\, |\delta \ue^{n,l}|^2=-\int_{\R^2}(\delta \ue^{n-1,l}\cdot \nabla \ue^{n})\cdot\delta \ue^{n,l}+\int_{\R^2}(a_\veps^{n+l}\, \nabla \delta \Pi_\veps^{n,l})\cdot \delta \ue^{n}
+\int_{\R^2}(\delta a_\veps^{n,l}\, \nabla  \Pi_\veps^{n})\cdot \delta \ue^{n,l}\, ,
\end{split}
\end{equation*} 
which implies 
\begin{equation*}
\begin{split}
\|\delta \ue^{n,l}(t)\|_{L^2}&\leq C \|\delta \vec u_{0,\veps}^{n,l}\|_{L^2}+C\int_0^t\|\nabla \ue^{n}\|_{L^\infty}\|\delta \ue^{n-1,l}\|_{L^2}+\|a_\veps^{n+l}\|_{L^\infty}\|\nabla \delta \Pi_\veps^{n,l}\|_{L^2}\, \detau\\
&+C\int_0^t\left(\|\delta \widetilde{a}_\veps^{n,l}(\tau )\|_{L^2}+\|\delta a_{0,\veps}^{n,l}\|_{L^\infty}\right)\|\nabla \Pi_\veps^{n}(\tau)\|_{L^2 \cap L^\infty} \, \detau\, ,
\end{split}
\end{equation*} 
where we have also employed the fact that $\delta a_\veps^{n,l} =\delta \widetilde{a}_\veps^{n,l}+\delta a_{0,\veps}^{n,l}$. 

Finally, for the pressure term we take the $\div$ operator in the momentum equation of system \eqref{syst_convergence_analysis}, obtaining 
\begin{equation*}
-\div \left((1+\veps a_\veps^{n+l})\frac{1}{\veps}\nabla \delta \Pi_\veps^{n,l}\right)=\div \left(-\delta \vu_{\veps}^{n,l} \cdot \nabla \vu_{\veps}^{n+l-1} +\delta \ue^{n-1,l} \cdot \nabla \ue^n + \frac{1}{\veps} (\delta \vu_{\veps}^{n,l})^\perp+\delta a_\veps^{n,l} \nabla \Pi_\veps^n \right),
\end{equation*} 
so that we have
\begin{equation}\label{Pi_cauchy}
\begin{split}
\|\nabla \delta \Pi_\veps^{n,l}\|_{L^2}&\leq C\veps \left(\|\delta
\ue^{n-1,l}\|_{L^2}\|\nabla \ue^n\|_{L^\infty}+\|\delta
a_\veps^{n,l}\, \nabla \Pi_\veps^n\|_{L^2}\right)\\
&+C\veps\|\delta
\ue^{n,l}\|_{L^2}\|\nabla \ue^{n+l-1}\|_{L^\infty}+C\|(\delta \ue^{n,l})^\perp\|_{L^2}\\
&\leq C\veps \left(\|\delta
\ue^{n-1,l}\|_{L^2}\|\nabla \ue^n\|_{L^\infty}+\|\delta
\widetilde{a}_\veps^{n,l}\|_{L^2}\|\nabla \Pi_\veps^n\|_{L^\infty}+\|\delta
a_{0,\veps}^{n,l}\|_{L^\infty}\|\nabla \Pi_\veps^n\|_{L^2}\right)\\
&+C\veps\|\delta
\ue^{n,l}\|_{L^2}\|\nabla \ue^{n+l-1}\|_{L^\infty} +C\|\delta \ue^{n,l}\|_{L^2}\, .
\end{split}
\end{equation}
At this point, applying Gr\"onwall lemma and using the bounds established in Paragraph \ref{ss:unif_est}, we thus argue that for $t\in [0,T^\ast]$:
\begin{equation*}
\begin{split}
\|\delta \widetilde{a}_\veps^{n,l}(t)\|_{L^2}+\|\delta \ue^{n,l}(t)\|_{L^2}&\leq C_{T^\ast}\left(\|\nabla \delta a_{0,\veps}^{n,l}\|_{L^\infty}+\| \delta a_{0,\veps}^{n,l}\|_{L^\infty}+\|\delta \vec u_{0,\veps}^{n,l}\|_{L^2}\right)\\
&+C_{T^\ast}\int_0^t\left(\|\delta \widetilde{a}_\veps^{n-1,l}(\tau)\|_{L^2}+\|\delta \ue^{n-1,l}(\tau)\|_{L^2}\right) \detau\, ,
\end{split}
\end{equation*}
where the constant $C_{T^\ast}$ depends on $T^\ast$ and on the initial data. 

After setting
\begin{equation*}
F_0^n:=\sup_{l\geq 0}\left(\|\nabla \delta a_{0,\veps}^{n,l}\|_{L^\infty}+\| \delta a_{0,\veps}^{n,l}\|_{L^\infty}+\|\delta \vec u_{0,\veps}^{n,l}\|_{L^2}\right)\quad \text{and}\quad G^n(t):=\sup_{l\geq 0}\sup_{[0,t]}\left(\|\delta \widetilde{a}_\veps^{n,l}\|_{L^2}+\|\delta \ue^{n,l}\|_{L^2}\right), 
\end{equation*}
by induction we may conclude that, for all $t\in [0,T^\ast]$, 
\begin{equation*}
\begin{split}
G^n(t)\leq C_{T^\ast}\sum_{k=0}^{n-1}\frac{(C_{T^\ast}T^\ast)^k}{k!}F_0^{n-k}+\frac{\left(C_{T^\ast}T^\ast\right)^n}{n!}G^0(t)\, .
\end{split}
\end{equation*}
Now, bearing \eqref{conv_in_data} in mind, we have 
\begin{equation*}
\lim_{n\rightarrow +\infty}F_0^n=0\, .
\end{equation*}
Hence, we may infer that
\begin{equation}\label{est_fin_conv}
\lim_{n\rightarrow +\infty}\sup_{l\geq 0}\sup_{t\in [0,T^\ast]}\left(\|\delta \widetilde{a}_\veps^{n,l}(t)\|_{L^2}+\|\delta \ue^{n,l}(t)\|_{L^2}\right)=0\, .
\end{equation}

Property \eqref{est_fin_conv} implies that both $(\widetilde{a}_\veps^n)_{n\in \N}$ and $(\ue^n)_{n\in \N}$ are Cauchy sequences in $C^0([0,T^\ast];L^2)$: therefore, such sequences converge to some functions $\widetilde{a}_\veps$ and $\ue$ in the same space. Taking advantage of previous computations in \eqref{Pi_cauchy}, we have also that $(\nabla \Pi_\veps^n)_{n\in \N}$ converge to a function $\nabla \Pi_\veps$ in $C^0([0,T^\ast];L^2)$.

Now, we define $a_\veps:= a_{0,\veps} +\widetilde{a}_\veps$. Hence, $ a_\veps -a_{0,\veps}$ is in $C^0([0,T^\ast];L^2)$.
Moreover, as $(\nabla a_\veps^n)_{n\in \N}$ is uniformly bounded in $L^\infty ([0,T^\ast];H^{s-1})$ and Sobolev spaces have the \textit{Fatou property}, we deduce that $\nabla a_\veps$ belongs to the same space. Moreover, since $(a^n_\veps )_{n\in \N}$ is uniformly bounded in $L^\infty ([0,T^\ast]\times \R^2)$, we also have that $a_\veps\in L^\infty ([0,T^\ast]\times \R^2)$. Analogously, as $(\ue^n)_{n\in \N}$ and $(\nabla \Pi_\veps^n)_{n\in \N}$ are uniformly bounded in $L^\infty ([0,T^\ast];H^s)$, we deduce that $\ue$ and $\nabla \Pi_\veps$ belong to $L^\infty ([0,T^\ast];H^s)$. 
 

Due to an interpolation argument, we see that the above sequences converge strongly in every intermediate $C^0([0,T^\ast];H^\sigma)$ for all $\sigma <s$. This is enough to pass to the limit in the equations satisfied by $(a_\veps^n,\ue^n,\nabla \Pi_\veps^n)_{n\in \N}$. Hence, $(a_\veps, \ue,\nabla \Pi_\veps)$ satisfies the original problem \eqref{Euler-a_eps_1}.

This having been established, we look at the time continuity of $a_\veps$. We exploit the transport equation:
$$\d_t a_{\veps}=-\vu_\veps \cdot \nabla a_\veps\, ,$$
noticing that the term on the right-hand side belongs to $L^\infty_{T^\ast}(L^\infty)$. Thus, we can deduce that $\d_t a_\veps \in L^\infty_{T^\ast}(L^\infty)$. Moreover, by embeddings, we already know that $\nabla a_\veps \in L^\infty_{T^\ast}(L^\infty)$. The previous two relations imply that $a_\veps \in W^{1,\infty}_{T^\ast}(L^\infty)\cap L^\infty_{T^\ast}(W^{1,\infty})$. That give us the desired regularity property $a_\veps \in C^0([0,T^\ast]\times \R^2)$. In addition, looking at the momentum equation in \eqref{Euler-a_eps_1} and employing Theorem \ref{thm_transport}, one obtains the claimed time regularity property for $\ue$. At this point, the time regularity for the pressure term $\nabla \Pi_\veps$ is recovered from the elliptic problem \eqref{eq:elliptic_problem}.

\subsection{Uniqueness}\label{ss:uniqueness_H^s}
We conclude this section showing the uniqueness of solutions in our framework.

We start by stating a uniqueness result, that is a consequence of a standard stability result based on energy methods. Since the proof is similar to the convergence argument of the previous paragraph, we will omit it (see e.g. \cite{D_JDE} for details). We recall that, in what follows, the parameter $\veps>0$ is fixed.
\begin{theorem} \label{th:uniq}
Let $\left(\vrho_\veps^{(1)}, \ue^{(1)}, \nabla \Pi_\veps^{(1)}\right)$ and $\left(\vrho_\veps^{(2)}, \ue^{(2)}, \nabla \Pi_\veps^{(2)}\right)$ solutions
to the Euler system \eqref{Euler_eps} associated to the initial data $\left(\vrho_{0,\veps}^{(1)}, \vec u_{0,\veps}^{(1)}\right)$ and $\left(\vrho_{0,\veps}^{(2)}, \vec u_{0,\veps}^{(2)}\right)$. Assume that, for some $T>0$, one has the following properties:
\begin{enumerate}[(i)]
\item the densities $\vrho_\veps^{(1)}$ and $\vrho_\veps^{(2)}$ are bounded and bounded away from zero;
\item the quantities $\delta \vrho_\veps:=\vrho_\veps^{(2)}-\vrho_\veps^{(1)}$ and $\delta \ue:=\ue^{(2)}-\ue^{(1)}$ belong to the space $C^1\big([0,T];L^2(\R^2)\big)$;
\item $\nabla \ue^{(1)}$, $\nabla \vrho_\veps^{(1)}$ and $\nabla \Pi_\veps^{(1)}$ belong to $ L^1\big([0,T];L^\infty(\R^2)\big)$.
\end{enumerate}

Then, for all $t\in[0,T]$, we have the stability inequality: 
\begin{equation}\label{stability_ineq}
\|\delta \vrho_\veps (t)\|_{L^2}+ \left\|\sqrt{\vrho_\veps^{(2)}(t)}\, \delta \ue (t)\right\|_{L^2}\leq \left(\|\delta \vrho_{0,\veps}\|_{L^2}+ \left\|\sqrt{\vrho_{0,\veps}^{(2)}}\, \delta \vec u_{0,\veps}\right\|_{L^2}\right)\, e^{CA(t)}\, ,
\end{equation}
for a universal constant $C>0$, where we have defined
\begin{equation*}\label{eq:def_A(t)}
A(t):=\int_0^t\left(\left\|\frac{\nabla \vrho_\veps^{(1)}}{\sqrt{\vrho_\veps^{(2)}}}\right\|_{L^\infty}+\left\|\frac{\nabla \Pi_\veps^{(1)}}{\vrho_\veps^{(1)}\sqrt{\vrho_\veps^{(2)}}}\right\|_{L^\infty}+\|\nabla \ue^{(1)}\|_{L^\infty}\right)\, \detau\, .
\end{equation*} 
\end{theorem}

It is worth to notice that, adapting the \textit{relative entropy} arguments presented
in Subsection 4.3 of \cite{C-F_RWA}, we can replace (in the statement above) the $C^1_T(L^2)$ requirement for $\delta \vrho_\veps$ and $\delta \ue$ with the $C^0_T
(L^2)$ regularity. However, one needs to pay an additional $L^2$ assumption on the densities. In this way, we will have a weak-strong uniqueness type result and we will prove it in the next theorem.

Concerning weak-strong results for density-dependent fluids, we refer to \cite{Ger}, where Germain exhibited a weak-strong uniqueness property within  a class of (weak) solutions to the compressible Navier-Stokes system satisfying a relative entropy inequality with respect to a (hypothetical) strong solution of the same problem (see also the work \cite{F-N-S} by Feireisl, Novotn\'y and Sun). 
Moreover, in \cite{F-J-N}, the authors established the weak-strong uniqueness property in the class of finite energy weak solutions, extending thus the classical results of Prodi \cite{Pr} and Serrin \cite{Ser} to the class of compressible fluid flows.

Before presenting the proof of the weak-strong uniqueness result, we state the definition of a \textit{finite energy weak solution} to system (2.1), such that $\vrho_{0,\veps}-1\in L^2(\R^2)$. We also recall that our densities have the form $\vrho_\veps=1+\veps R_\veps$.
\begin{definition}\label{weak_sol_L^2}
Let $T>0$ and $\veps \in \; ]0,1]$ be fixed. Let $(\vrho_{0,\veps},\vec u_{0,\veps})$ be an initial datum fulfilling the assumptions in Paragraph \ref{ss:FormProb}.
We say that $(\vrho_\veps, \ue)$ is a \textit{finite energy weak solution} to system \eqref{Euler_eps} in $[0,T]\times \R^2$, related to the previous initial datum, if:
\begin{itemize}
\item $\vrho_\veps \in L^\infty([0,T]\times \R^2)$ and $\vrho_\veps -1\in C^0([0,T];L^2(\R^2))$;
\item $\ue \in L^\infty([0,T];L^2(\R^2))\cap C_w^0([0,T];L^2(\R^2))$;
\item the mass equation is satisfied in the weak sense: 
\begin{equation*}
\int_0^T\int_{\R^2}\Big(\vrho_\veps \, \d_t \varphi+ \vrho_\veps \ue \cdot \nabla \varphi\Big) \, \dxdt +\int_{\R^2}\vrho_{0,\veps}\varphi(0, \cdot )\dx =\int_{\R^2}\vrho_\veps (T)\varphi(T,\cdot)\dx \, ,
\end{equation*}
for all $\varphi\in C_c^\infty ([0,T]\times \R^2;\R)$;
\item the divergence-free condition $\div \ue =0$ is satisfied in $\mathcal{D}^\prime (]0,T[\, \times \R^2)$;
\item the momentum equation is satisfied in the weak sense:
\begin{equation*}
\begin{split}
\int_0^T\int_{\R^2}\Big(\vrho_\veps \ue \cdot \d_t \vec \psi +[\vrho_\veps\ue \otimes \ue] :\nabla \vec \psi - \frac{1}{\veps}\vrho_\veps \ue^\perp \cdot \vec \psi \Big) \dxdt+ \int_{\R^2}\vrho_{0,\veps}& \vec u_{0,\veps}\cdot \vec \psi (0)\dx \\
&=\int_{\R^2}\vrho_\veps(T)\vec u_\veps(T)\vec \psi(T)\dx,
\end{split}
\end{equation*}
for any $\vec \psi \in C_c^\infty([0,T]\times \R^2;\R^2)$ such that $\div \vec \psi=0$;
\item for almost every $t\in [0,T]$, the two following energy balances hold true:
\begin{equation*}
\int_{\R^2} \vrho_\veps (t) |\ue (t)|^2\dx\leq \int_{\R^2} \vrho_{0,\veps} |\vec u_{0,\veps}|^2\dx \quad \text{and}\quad \int_{\R^2}(\vrho_{\veps}-1)^2\dx\leq \int_{\R^2}(\vrho_{0,\veps}-1)^2\dx\, .
\end{equation*}
\end{itemize}
\end{definition}

\begin{theorem} \label{th:uniq_bis}
Let $\veps\in \, ]0,1]$ be fixed. Let $\left(\vrho_\veps^{(1)}, \ue^{(1)}\right)$ and $\left(\vrho_\veps^{(2)}, \ue^{(2)}\right)$ finite energy weak solutions
to the Euler system \eqref{Euler_eps} as in Definition \ref{weak_sol_L^2} with initial data $\left(\vrho_{0,\veps}^{(1)}, \vec u_{0,\veps}^{(1)}\right)$ and $\left(\vrho_{0,\veps}^{(2)}, \vec u_{0,\veps}^{(2)}\right)$. Assume that, for some $T>0$, one has the following properties:
\begin{enumerate}[(i)]
\item $\nabla \ue^{(1)}$ and $\nabla R_\veps^{(1)}$ belong to $ L^1\big([0,T];L^\infty(\R^2)\big)$;
\item $\nabla \Pi_\veps^{(1)}$ is in $ L^1\big([0,T];L^\infty(\R^2)\cap L^2(\R^2)\big)$.
\end{enumerate}

Then, for all $t\in[0,T]$, we have the stability inequality \eqref{stability_ineq}. 
\end{theorem}
\begin{proof} We start by defining, for $i=1,2$:
\begin{equation*}
R_\veps^{(i)}:=\frac{\vrho_\veps^{(i)}-1}{\veps}\quad \text{and}\quad R_{0,\veps}^{(i)}:=\frac{\vrho_{0,\veps}^{(i)}-1}{\veps}\, ,
\end{equation*}
and we notice that, owing to the continuity equation in \eqref{Euler_eps} and the divergence-free condition $\div \vec u_\veps^{(i)}=0$, one has
\begin{equation}\label{ws:transport_R}
\d_tR_\veps^{(i)}+\div (R_\veps^{(i)}\vec u_\veps^{(i)})=0 \quad \text{with}\quad R_\veps^{(i)}(0)=R_{0,\veps}^{(i)}.
\end{equation}
 
For simplicity of notation, we fix $\veps=1$ throughout this proof and let us assume for a while the couple $(R^{(1)},\vec u^{(1)})$ of smooth functions such that $R^{(1)},\vec u^{(1)}\in C^\infty_c(\R_+\times \R^2)$ and $\div \vec u^{(1)}=0$, with the support of $R^{(1)}$ and $\vec u^{(1)}$ included in $[0,T]\times \R^2$.
First of all, we use $\vec u^{(1)}$ as a test function in the weak formulation of the momentum equation, finding that 
\begin{equation}\label{mom-eq-reg-1}
\begin{split}
\int_{\R^2} \vrho^{(2)}(T)\vec u^{(2)}(T)\cdot \vec u^{(1)}(T) \dx&= \int_{\R^2}\vrho_0^{(2)}\vec u_0^{(2)}\cdot \vec u_0^{(1)} \dx +\int_0^T\int_{\R^2}\vrho^{(2)}\vec u^{(2)}\cdot \d_t\vec u^{(1)} \dxdt\\
&+\int_0^T\int_{\R^2}(\vrho^{(2)}\vec u^{(2)}\otimes \vec u^{(2)}):\nabla \vec u^{(1)} \dxdt+\int_0^T\int_{\R^2}\vrho^{(2)}\vec u^{(2)}\cdot (\vec u^{(1)})^\perp \dxdt\, ,
\end{split}
\end{equation} 
where we have also noted that $ (\vec u^{(2)})^\perp\cdot \vec u^{(1)}=- \vec u^{(2)}\cdot (\vec u^{(1)})^\perp$.

Next, testing the mass equation against $|\vec u^{(1)}|^2/2$, we obtain
\begin{equation}\label{mom-eq-reg-2}
\begin{split}
\frac{1}{2}\int_{\R^2}\vrho^{(2)}|\vec u^{(1)}|^2\dx&=\frac{1}{2}\int_{\R^2}\vrho_0^{(2)}|\vec u_0^{(1)}|^2\dx +\int_0^T\int_{\R^2}\vrho^{(2)}\vec u^{(1)}\cdot \d_t \vec u^{(1)} \dxdt\\
&+\frac{1}{2}\int_0^T\int_{\R^2}\vrho^{(2)}\vec u^{(2)}\cdot \nabla |\vec u^{(1)}|^2 \dxdt\\
&=\frac{1}{2}\int_{\R^2}\vrho_0^{(2)}|\vec u_0^{(1)}|^2\dx +\int_0^T\int_{\R^2}\vrho^{(2)}\vec u^{(1)}\cdot \d_t \vec u^{(1)} \dxdt\\
&+\frac{1}{2}\int_0^T\int_{\R^2}(\vrho^{(2)}\vec u^{(2)}\otimes\vec u^{(1)}): \nabla \vec u^{(1)} \dxdt.
\end{split}
\end{equation}
Recall also that the energy inequality reads
\begin{equation*}
\frac{1}{2}\int_{\R^2}\vrho^{(2)}|\vec u^{(2)}|^2\dx \leq\dfrac{1}{2}\int_{\R^2}\vrho_0^{(2)}|\vec u_0^{(2)}|^2\dx \, .
\end{equation*}
Now, we take care of the density oscillations $R^{(i)}$. We test the transport equation \eqref{ws:transport_R} for $R^{(2)}$ against $R^{(1)}$, getting 
\begin{equation}\label{transp-eq-reg-1}
\int_{\R^2}R^{(2)}(T)R^{(1)}(T)\dx=\int_{\R^2}R_0^{(2)}R_0^{(1)}\dx +\int_0^T\int_{\R^2}R^{(2)}\d_tR^{(1)}\dxdt+\int_0^T\int_{\R^2}R^{(2)}\vec u^{(2)}\cdot \nabla R^{(1)}\dxdt.
\end{equation}
Recalling Definition \ref{weak_sol_L^2}, we have the following energy balance:
\begin{equation*}
\int_{\R^2}|R^{(2)}(T)|^2\dx \leq\int_{\R^2}|R^{(2)}_0|^2\dx\, .
\end{equation*}
At this point, testing $\d_t1+\div(1\, \vec u^{(2)})=0$ against $|R^{(1)}|^2/2$, we may infer that 
\begin{equation}\label{transp-eq-reg-2}
\frac{1}{2}\int_{\R^2} |R^{(1)}|^2\dx=\frac{1}{2}\int_{\R^2}|R_0^{(1)}|^2 \dx +\int_0^T\int_{\R^2}R^{(1)}\d_tR^{(1)}\dxdt+\int_0^T\int_{\R^2}R^{(1)}\vec u^{(2)}\cdot \nabla R^{(1)}\dxdt\, .
\end{equation}
Now, for notational convenience, let us define 
\begin{equation*}
\delta R:=R^{(2)}-R^{(1)}\quad \text{and}\quad \delta \vec u:=\vec u^{(2)}-\vec u^{(1)}\, .
\end{equation*} 
Putting all the previous relations together, we obtain 
\begin{equation}\label{rel_en_ineq}
\begin{split}
\frac{1}{2}\int_{\R^2}\Big(\vrho^{(2)}(T)|\delta \vec u(T)|^2+|\delta R(T)|^2\Big) \dx&\leq \frac{1}{2}\int_{\R^2}\Big(\vrho_0^{(2)}|\delta \vec u_0|^2+|\delta R_0|^2\Big) \dx-\int_0^T\int_{\R^2}\vrho^{(2)}\vec u^{(2)}\cdot (\vec u^{(1)})^\perp\dxdt\\
&-\int_0^T\int_{\R^2}\Big(\vrho^{(2)}\delta \vec u\cdot \d_t\vec u^{(1)}+\delta R\cdot \d_tR^{(1)}\Big)\dxdt \\
&-\int_0^T\int_{\R^2}\Big(\vrho^{(2)}\vec u^{(2)}\otimes \delta \vec u :\nabla \vec u^{(1)}+\delta R \, \vec u^{(2)}\cdot \nabla R^{(1)}\Big)\dxdt\, .
\end{split}
\end{equation}
Next, we remark that we can write 
$$ \vrho^{(2)}\vec u^{(2)}\otimes \delta \vec u :\nabla \vec u^{(1)}=\vrho^{(2)}\vec u^{(2)}\cdot \nabla \vec u^{(1)}\cdot \delta \vec u $$
and that we have $\vec u^{(2)}\cdot (\vec u^{(1)})^\perp=\delta \vec u\cdot (\vec u^{(1)})^\perp$ by orthogonality. 

Therefore, relation \eqref{rel_en_ineq} can be recasted as
\begin{equation*}\label{rel_en_ineq_1}
\begin{split}
\frac{1}{2}\int_{\R^2}\Big(\vrho^{(2)}(T)|\delta \vec u(T)|^2+|\delta R(T)|^2\Big) \dx&\leq \frac{1}{2}\int_{\R^2}\Big(\vrho_0^{(2)}|\delta \vec u_0|^2+|\delta R_0|^2\Big) \dx\\
&-\int_0^T\int_{\R^2}\vrho^{(2)}\Big(\d_t \vec u^{(1)}+\vec u^{(2)}\cdot \nabla \vec u^{(1)}+(\vec u^{(1)})^\perp\Big)\cdot \delta \vec u\, \dxdt \\
&-\int_0^T\int_{\R^2}\Big(\d_tR^{(1)}+\vec u^{(2)}\cdot \nabla R^{(1)}\Big)\cdot \delta R\, \dxdt\, .
\end{split}
\end{equation*}
At this point, we add and subtract the quantities $\pm \vrho^{(2)}\vec u^{(1)}\cdot \nabla \vec u^{(1)}\cdot \, \delta \vec u\pm \vrho^{(2)}\frac{1}{\vrho^{(1)}}\nabla \Pi^{(1)}\cdot \, \delta \vec u$ and $\pm \vec u^{(1)}\cdot \nabla R^{(1)}\cdot \, \delta R$, yielding 
\begin{equation}\label{rel_en_ineq_2}
\begin{split}
\frac{1}{2}\int_{\R^2}\Big(\vrho^{(2)}(T)|\delta \vec u(T)|^2+|\delta R(T)|^2\Big) \dx&\leq \frac{1}{2}\int_{\R^2}\Big(\vrho_0^{(2)}|\delta \vec u_0|^2+|\delta R_0|^2\Big) \dx\\
&-\int_0^T\int_{\R^2}\Big(\vrho^{(2)}\delta \vec u\cdot \nabla \vec u^{(1)} +\delta R\frac{1}{\vrho^{(1)}}\nabla \Pi^{(1)}\Big)\cdot \delta \vec u \, \dxdt  \\
&-\int_0^T\int_{\R^2}\delta \vec u\cdot R^{(1)}\cdot \delta R \dxdt\, ,
\end{split}
\end{equation}
where we have used the fact that $\vec u^{(1)}$ is solution to the Euler system and $\int_{\R^2}\nabla \Pi^{(1)}\cdot \delta \vec u \, \dx=0$. 

Therefore, setting $\mc E(T):=\|\sqrt{\vrho^{(2)}(T)}\delta \vec u(T)\|^2_{L^2}+\|\delta R(T)\|^2_{L^2}$, from relation \eqref{rel_en_ineq_2} we can deduce that
\begin{equation*}\label{rel_en_inq_3}
\mc E(T)\leq \mc E(0)+\int_0^T\left(\|\nabla \vec u^{(1)}\|_{L^\infty}+\left\|\frac{1}{\sqrt{\vrho^{(2)}}\vrho^{(1)}}\nabla \Pi^{(1)}\right\|_{L^\infty}+\left\|\frac{1}{\sqrt{\vrho^{(2)}}}\nabla R^{(1)}\right\|_{L^\infty}\right)\mc E(t) \dt\, .
\end{equation*}
An application of Gr\"onwall lemma yields the desired stability inequality \eqref{stability_ineq}. 

In order to get the result, having the regularity stated in the theorem, we argue by density. 

Thanks to the regularity (stated in Definition \ref{weak_sol_L^2}) of weak solutions to the Euler equations \eqref{Euler_eps} and assumption $(i)$ of the theorem, all the terms appearing in relations \eqref{mom-eq-reg-1} and \eqref{mom-eq-reg-2} are well-defined, if in addition we have $\d_t \vec u^{(1)}\in L^1_T(L^2)$. However, this condition on the time derivative of the velocity field $\vec u^{(1)}$ comes free from the momentum equation 
\begin{equation}\label{vec_1}
 \d_t \vec u^{(1)}=-\left(\vec u^{(1)}\cdot \nabla \vec u^{(1)}+(\vec u^{(1)})^\perp + \frac{1}{\vrho^{(1)}}\nabla \Pi^{(1)}\right)\, . 
 \end{equation}
Since $\vec u^{(1)}\in L^\infty_T(L^2)$ with $\nabla \vec u^{(1)}\in L^1_T(L^\infty)$ and $\vrho^{(1)}\in L^\infty_T(L^\infty)$, condition $(ii)$ implies that the right-hand side of \eqref{vec_1} is in $L^1_T(L^2)$. Recalling the regularity in Definition \ref{weak_sol_L^2} of $\vec u^{(1)}$, one gets $\vec u^{(1)}\in W^{1,1}_T(L^2)$ and hence $\vec u^{(1)}\in C^0_T(L^2)$. 

Analogously in order to justify computations in \eqref{transp-eq-reg-1} and \eqref{transp-eq-reg-2}, besides the previous regularity conditions, one needs the additional assumption $\d_t R^{(1)}\in L^1_T(L^2)$. Once again, one can take advantage of the continuity equation \eqref{ws:transport_R} to obtain the required regularity for $\d_t R^{(1)}$. Finally, condition $(ii)$ is necessary to make sense of relation \eqref{rel_en_ineq_2}.

This concludes the proof of the theorem.

\end{proof}

\section{Asymptotic analysis} \label{s:sing-pert}

The main goal of this section is to show the convergence when $\veps\rightarrow 0$: we achieve it employing a \textit{compensated compactness} technique. We point out that, in the sequel, the time $T>0$ is fixed by the existence theory developed in Section \ref{s:well-posedness_original_problem}.

We will show that \eqref{full Euler} converges towards a limit system, represented by the quasi-homogeneous incompressible Euler equations:
\begin{equation}\label{QH-Euler_syst}
\begin{cases}
\d_tR+\vu \cdot \nabla R=0\\
\d_t \vu +\vu \cdot \nabla \vu +R\vu^\perp +\nabla \Pi=0\\
\div \vu=0\, .
\end{cases}
\end{equation}
The previous system consists of a transport equation for the quantity $R$ (that can be interpreted as the deviation with respect to the constant density profile) and an Euler type equation for the limit velocity field $\vec u$.

Nowadays, the strategy to tackle this family of problems (called \textit{singular perturbation problems}) is somehow standard.   
First of all, one has to show existence of $(\vrho_\veps, \ue, \nabla \Pi_\veps)$ for any value $\veps >0$ fixed and to find uniform (with respect to $\veps$) bounds for the sequence $(\vrho_\veps, \ue, \nabla \Pi_\veps)_\veps$. This analysis was performed in the previous Section \ref{s:well-posedness_original_problem}. Next, thanks to the uniform bounds, one can extract weak limit points, for which one has to find some constraints: the singular terms have to vanish at the limit. This is done in Subsection \ref{sss:unif-prelim}.

 Finally, after performing the \textit{compensated compactness} arguments, one can describe the limit dynamics (see Paragraph \ref{ss:wave_system} below).

The choice of using this technique derives by the fact that the oscillations in time of the solutions are out of control (see Subsection \ref{ss:wave_system}). To overcome this issue, rather than employing the standard $H^s$ estimates, we take advantage of the weak formulation of the problem. We test the equations against divergence-free test functions: this will lead to useful cancellations. In particular, we avoid to study the pressure term.
At the end, we close the argument by noticing that the weak limit solutions are actually regular solutions. 



\subsection{Preliminaries and constraint at the limit} \label{sss:unif-prelim}
We start this subsection by recalling the uniform bounds developed in Section \ref{s:well-posedness_original_problem}. The fluctuations $R_{\veps}$ satisfy the controls
	\begin{align*}
&\sup_{\veps\in\,]0,1]}\left\|  R_{\veps} \right\|_{L^\infty(\R^2)}\,\leq \, C\\
&\sup_{\veps\in\,]0,1]}\left\| \nabla R_{\veps} \right\|_{H^{s-1}(\R^2)}\,\leq \, C \, ,
	\end{align*}
where $R_\veps:=(\vrho_\veps-1)/\veps$ as above.

As for the velocity fields, we have obtained the following uniform bound:
\begin{equation*}
\sup_{\veps\in\,]0,1]}\left\|  \vu_{\veps} \right\|_{H^s(\R^2)}\,\leq \, C\, .
\end{equation*}
Thanks to the previous uniform estimates, we can assume (up to passing to subsequences) that there exist $R \in W^{1,\infty}(\R^2)$, with $\nabla R\in H^{s-1}(\R^2)$, and $\vec u \in H^s(\R^2)$ such that 
\begin{equation}\label{limit_points}
\begin{split}
R:=\lim_{\veps \rightarrow 0}R_{\veps} \quad &\text{in}\quad L^\infty (\R^2) \\
\nabla R:=\lim_{\veps \rightarrow 0}\nabla R_{\veps}\quad &\text{in}\quad H^{s-1} (\R^2)\\
\vu:=\lim_{\veps \rightarrow 0}\vu_{\veps}\quad &\text{in}\quad H^s (\R^2)\, ,
\end{split}
\end{equation} 
where we agree that the previous limits are taken in the corresponding weak-$\ast$ topology.

\begin{remark}\label{rmk:conv_rho_1}
It is evident that $\vrho_\veps -1=O(\veps)$ in $L^\infty_T (L^{\infty})$ and therefore that $\vrho_\veps \ue$ weakly-$\ast$ converge to $\vec  u$ e.g. in the space $L^\infty_T( L^2)$.
\end{remark}

Next, we notice that the solutions stated in Theorem \ref{W-P_fullE} are \textit{strong} solutions. In particular, they satisfy in a weak sense the mass equation and the momentum equation, respectively:  
\begin{equation}\label{weak-con}
	-\int_0^T\int_{\R^2} \left( \vre \partial_t \varphi  + \vre\ue \cdot \nabla_x \varphi \right) \dxdt = 
	\int_{\R^2} \vrez \varphi(0,\cdot) \dx
	\end{equation}
for any $\varphi\in C^\infty_c([0,T[\,\times \R^2;\R)$;
	\begin{align}
	&\int_0^T\!\!\!\int_{\R^2}  
	\left( - \vre \ue \cdot \partial_t \vec\psi - \vre [\ue\otimes\ue]  : \nabla_x \vec\psi 
	+ \frac{1}{\ep} \,  \vre \ue^\perp \cdot \vec\psi  \right) \dxdt 
	= \int_{\R^2}\vrez \uez  \cdot \vec\psi (0,\cdot) \dx \label{weak-mom} 
	\end{align}
for any test function $\vec\psi\in C^\infty_c([0,T[\,\times \R^2; \R^2)$ such that $\div \vec\psi=0$;
 
Moreover, the divergence-free condition on $\ue$ is satisfied in $\mc D^\prime (]0,T[\times \R^2)$.

Before going on, in the following lemma, we 
characterize the limit for the quantity $R_\veps \ue$. We recall that $R_\veps$ satisfies 
\begin{equation}\label{eq_transport_R_veps}
\d_t R_\veps =-\div (R_\veps \ue)\, , \quad \quad (R_\veps)_{|t=0}=R_{0,\veps}.
\end{equation}  
\begin{lemma}\label{l:reg_Ru_veps}
Let $(R_\veps)_\veps$ be uniformly bounded in $L^\infty_T(L^\infty(\R^2))$ with $(\nabla R_\veps)_\veps\subset L^\infty_T(H^{s-1}(\R^2))$, and let the velocity fields $(\ue)_\veps$ be uniformly bounded in $L^\infty(H^s(\R^2))$. Moreover, for any $\veps \in\; ]0,1]$, assume that the couple $(R_\veps, \vec u_\veps)$ solves the transport equation \eqref{eq_transport_R_veps}. Let $(R, \vec u)$ be the limit point identified in \eqref{limit_points}. Then, up to an extraction:
\begin{enumerate}
\item[(i)] $R_\veps \rightarrow R$ in $C^0_T(C^0_{\rm loc}(\R^2))$;
\item[(ii)] the product $R_\veps \ue$ converges to $R\vec u$ in the distributional sense.
\end{enumerate}
\end{lemma}

\begin{proof}
We look at the transport equation \eqref{eq_transport_R_veps} for $R_\veps$. 
We employ Proposition \ref{prop:app_fine_tame_est} on the term in the right-hand side, obtaining 
\begin{equation*}
\|R_\veps \ue\|_{H^s}\leq C\left(\|R_\veps\|_{L^\infty}\|\ue  \|_{H^s}+\|\nabla R_\veps\|_{H^{s-1}}\|\ue\|_{L^\infty}\right)\, .
\end{equation*}
By embeddings, this implies that the sequence $(\d_t R_\veps)_\veps$ is uniformly bounded e.g. in $L^\infty_T (L^\infty)$ and so $(R_\veps)_\veps$ is bounded in $W^{1,\infty}_T(L^\infty)$ uniformly in $\veps$. On the other hand, we know that $(\nabla R_{\veps})_\veps$ is bounded in $L^\infty_T (L^\infty)$. Then, by Ascoli-Arzel\`a theorem, we gather that the family $(R_\veps)_\veps$ is compact in e.g. $C^0_T (C_{\rm loc}^{0})$ and hence we deduce the strong convergence property, up to passing to a suitable subsequence (not relabeled here), 
\begin{equation*}
R_\veps \rightarrow R \quad \text{in} \quad C^0([0,T]\, ;C_{\rm loc}^{0})\, .
\end{equation*}
Finally, since $(\ue)_\veps$ is weakly-$\ast$ convergent e.g. in $L^\infty_T( L^2)$ to $\vec u$, we get $R_\veps \ue \weakstar R \vec  u$ in the space $L^\infty_T( L_{\rm loc}^2)$.
\end{proof}


At this point, as anticipated in the introduction of this section, we have to highlight the constraint that the limit points have to satisfy. We have to point out that this condition does not fully characterize the limit dynamics (see Subsection \ref{ss:wave_system} below). 

The only singular term present in the equations is the Coriolis force. Then, we test the momentum equation in \eqref{weak-mom} 
against $\veps \vec \psi$ with $\vec \psi \in C_c^\infty([0,T[\, \times \R^2;\R^2)$ such that $\div \vec \psi =0$.
Keeping in mind the assumptions on the initial data and due to the fact that $(\vrho_\veps \ue )_\veps$ is uniformly bounded in e.g. $L^\infty_T(L^2) $ and so is $(\vrho \ue \otimes \ue)_\veps$ in $L^\infty_T(L^1)$, it follows that all the terms in equation \eqref{weak-mom}, apart from the Coriolis operator, go to $0$ in the limit for $\veps\rightarrow 0$.

Therefore, 
we infer that, for any $\vec \psi \in C_c^\infty([0,T[\, \times \R^2;\R^2)$ such that $\div \vec \psi =0$,
\begin{equation*}
\lim_{\veps \rightarrow 0}\int_0^T\int_{\R^2}\vrho_\veps \ue^\perp \cdot \vec \psi \dx \dt=\int_0^T \int_{\R^2}\vec u^\perp \cdot \vec \psi \dx \dt=0\, .
\end{equation*}
This property tells us that $\vec u^\perp =\nabla \pi$, for some suitable function $\pi$. 

However, this relation does \textit{not} add more information on the limit dynamics, since we already know that the divergence-free condition $\div \ue=0$ is satisfied for all $\veps>0$. 


\subsection{Wave system and convergence}\label{ss:wave_system}
The goal of the present subsection is to describe oscillations of solutions to show convergence to the limit system. The Coriolis term is responsible for fast oscillations in time of solutions, which may prevent the convergence. To overcome this issue we implement a strategy based on \textit{compensated compactness} arguments. Namely, we perform algebraic manipulations on the wave system (see \eqref{wave system} below), in order to derive 
compactness properties for the quantity $\gamma_\veps:=\curl (\vrho_\veps \ue)$. This will be enough to pass to the limit in the momentum equation (and, in particular, in the convective term). 

Let us define 
\begin{equation*}
\vec{V}_\veps:=\vrho_\veps \ue 
\, ,
\end{equation*}
that is uniformly bounded in $L^\infty_T(H^s)$, due to Proposition \ref{prop:app_fine_tame_est} in the Appendix. 

Now, using the fact that $\vrho_\veps=1+\veps R_\veps$, we recast the continuity equation in the following way:
\begin{equation}\label{eq:wave_mass}
\veps\d_t R_\veps+\div \vec V_\veps=0\, .
\end{equation}
In light of the uniform bounds and convergence properties stated in Lemma \ref{l:reg_Ru_veps}, we can easily pass to the limit in the previous formulation (or rather in \eqref{eq_transport_R_veps}) finding  
\begin{equation}\label{eq:limit_transp_R}
\d_t R+\div(R \vec u)=0\, .
\end{equation}
At this point, we decompose
$$ \vrho_\veps \ue^\perp=\ue^\perp +\veps\, R_\veps \ue^\perp $$
and from the momentum equation one can deduce
\begin{equation}\label{eq:wave_mom}
\veps \d_t \vec V_\veps +\nabla \Pi_\veps +\ue^\perp= \veps f_\veps
\end{equation}
where we have defined 
\begin{equation}\label{def_f}
f_\veps:=-\div (\vrho_\veps \ue \otimes \ue)-R_\veps \ue^\perp\, .
\end{equation}

In this way, we can rewrite system \eqref{Euler_eps} in the wave form
\begin{equation}\label{wave system}
\begin{cases}
\veps\d_t R_\veps+\div \vec V_\veps=0\\
\veps \d_t \vec V_\veps +\nabla \Pi_\veps +\ue^\perp= \veps f_\veps\, .
\end{cases}
\end{equation}

Applying again Proposition \ref{prop:app_fine_tame_est}, one can show that the terms $\vrho_\veps \ue \otimes \ue$ and $R_\veps \ue^\perp$ are uniformly bounded in $L^\infty_T(H^s)$. Thus, it follows that $(f_\veps)_\veps \subset L^\infty_T(H^{s-1})$.


However, the uniform bounds in Section \ref{s:well-posedness_original_problem} are not enough for proving convergence in the weak formulation of the momentum equation. Indeed, on the one hand, those controls allow to pass to the limit in the $\d_t$ term and in the initial datum; on the other hand, the non-linear term and the Coriolis force are out of control. We postpone the convergence analysis of the Coriolis force in the next Paragraph \ref{ss:limit} and now we focus on the 
the convective term $\div (\vrho_\veps \ue \otimes \ue)$ in \eqref{weak-mom}. We proceed as follows: first of all, we reduce our study to the constant density case (see Lemma \ref{l:approx_convective_term} below). Next, we apply the \textit{compensated compactness} argument.

\begin{lemma}\label{l:approx_convective_term}
Let $T>0$. For any test function $\vec \psi \in C_c^\infty([0,T[\times \R^2;\R^2)$, we get 
\begin{equation}\label{eq_approx_convective_term}
\limsup_{\veps\rightarrow 0}\left|\int_0^T \int_{\R^2}\vrho_\veps \ue\otimes \ue :\nabla \vec \psi \, \dxdt- \int_0^T\int_{\R^2} \ue \otimes \ue :\nabla \vec \psi \, \dxdt\right|=0\, .
\end{equation}
\end{lemma}
\begin{proof}
Let $\psi \in C_c^\infty([0,T[\times \R^2;\R^2)$ with $\Supp \vec \psi \subset [0,T]\times K$ for some compact $K\subset \R^2$. Therefore, we can write
\begin{equation*}
\int_0^T \int_{K}\vrho_\veps \ue\otimes \ue :\nabla \vec \psi \, \dxdt= \int_0^T\int_{K} \ue \otimes \ue :\nabla \vec \psi \, \dxdt+\veps \int_0^T\int_{K} R_\veps \ue \otimes \ue :\nabla \vec \psi \, \dxdt\, .
\end{equation*}
As a consequence of the uniform bounds e.g. $(\ue)_\veps \subset L^\infty_T(H^s)$ and $(R_\veps)_\veps\subset L_T^\infty(L^\infty)$, the second integral in the right-hand side is of order $\veps$.
\end{proof}

Thanks to Lemma \ref{l:approx_convective_term}, we are reduced to study the convergence (with respect to $\veps$) of the integral
\begin{equation*}
-\int_0^T\int_{\R^2} \ue \otimes \ue :\nabla \vec \psi \, \dxdt=\int_0^T\int_{\R^2} \div(\ue \otimes \ue) \cdot \vec \psi \, \dxdt\, .
\end{equation*}
Owing to the divergence-free condition we can write:
\begin{equation}\label{eq:conv_term_rel}
\div (\ue \otimes \ue )=\ue \cdot \nabla \ue =\frac{1}{2}\nabla |\ue|^2+\omega_\veps \, \ue^\perp\, ,
\end{equation}
where we have denoted $\omega_\veps:=\curl \ue$.

Notice that the former term, since it is a perfect gradient, vanishes identically when tested against $\vec \psi$ such that $\div \vec \psi=0$. 
As for the latter term we take advantage of equation \eqref{eq:wave_mom}. Taking the $\curl$ we get
\begin{equation}\label{relation_gamma}
\d_t  \gamma_\veps=\curl f_\veps \, ,
\end{equation}
where we have set $\gamma_\veps:= \curl \vec V_\veps$ with $\vec V_\veps:= \vrho_\veps \ue$. We recall also that $f_\veps$ defined in \eqref{def_f} is uniformly bounded in the space $L^\infty_T(H^{s-1})$.
Then, relation \eqref{relation_gamma} implies that the family $(\d_t \gamma_\veps)_\veps$ is uniformly bounded in $L^\infty_T(H^{s-2})$. As a result, we get $(\gamma_\veps)_\veps \subset W^{1,\infty}_T(H^{s-2})$. On the other hand, the sequence $(\nabla \gamma_\veps) _\veps$ is also uniformly bounded in $L^\infty_T(H^{s-2})$.
At this point, the Ascoli-Arzel\`a theorem gives compactness of $(\gamma_\veps)_\veps$ in e.g. $C^0_T(H^{s-2}_{\rm loc})$. Then, it converges (up to extracting a subsequence) to a tempered distribution $\gamma$ in the same space. Thus, it follows that 
\begin{equation*}
\gamma_\veps \longrightarrow \gamma \quad \text{in}\quad \mathcal{D}^\prime (\R_+\times \R^2)\, .
\end{equation*}
But since we already know the convergence $\vec{V}_\veps:=\vrho_\veps \ue \weakstar \vec u$ in e.g. $L^\infty_T(L^2)$, it follows that $\gamma_\veps:= \curl \vec V_\veps \weak \omega:=\curl \vec u$ in $\mc D^\prime$, hence $\gamma =\curl \vec u=\omega$.

Finally, writing $\vrho_\veps=1+\veps R_\veps$, we obtain 
$$ \gamma_\veps:=\curl (\vrho_\veps \ue)=\omega_\veps+\veps \curl (R_\veps \ue)\, , $$
where the family $(\curl (R_\veps \ue))_\veps$ is uniformly bounded in $L^\infty_T(H^{s-1})$. From this relation and the previous analysis, we deduce the strong convergence (up to an extraction) for $\veps\rightarrow 0$:
\begin{equation*}
\omega_\veps\longrightarrow \omega \qquad \text{ in }\qquad L^\infty_T(H^{s-2}_{\rm loc})\,. 
\end{equation*}


In the end, we have proved the following convergence result for the convective term $\div (\vec u_\veps \otimes \ue)$. 
\begin{lemma}\label{lemma:limit_convective_term}
Let $T>0$. Up to passing to a suitable subsequence, one has the following convergence for $\veps\rightarrow 0$:
\begin{equation}\label{limit_convective_term}
\int_0^T\int_{\R^2} \ue \otimes \ue :\nabla \vec \psi \, \dxdt\longrightarrow \int_0^T\int_{\R^2} \omega\, \vu^\perp \cdot \vec \psi\, \dxdt\, ,
\end{equation}
for any test function $\vec \psi \in C_c^\infty([0,T[\times \R^2;\R^2)$ such that $\div \vec \psi=0$.
\end{lemma}

\subsection{The limit system} \label{ss:limit} 
With the convergence established in Paragraph \ref{ss:wave_system}, we can pass to the limit in the momentum equation.

To begin with, we take a test-function $\vec\psi$ such that 
\begin{equation}\label{def:test_function}
\vec \psi =\nabla^\perp \varphi\quad \quad  \text{with}\quad \quad
 \varphi \in C_c^\infty ([0,T[\times \R^2;\R)\, .
\end{equation}

For such a $\vec\psi$, all the gradient terms vanish identically. First of all, we recall the momentum equation in its weak formulation:
\begin{equation}\label{weak_mom_limit}
\int_0^T\!\!\!\int_{\R^2}  
	\left( - \vre \ue \cdot \partial_t \vec\psi - \vre [\ue\otimes\ue]  : \nabla_x \vec\psi 
	+ \frac{1}{\ep} \,  \vre \ue^\perp \cdot \vec\psi  \right) \dxdt 
	= \int_{\R^2}\vrez \uez  \cdot \vec\psi (0,\cdot) \dx\, .
\end{equation}
Making use of the uniform bounds, we can pass to the limit in the $\d_t$ term and thanks to our assumptions and embeddings we have $\vrho_{0,\veps}\vu_{0,\veps}\weak \vu_0$ in e.g. $L_{\rm loc}^{2}$.

For the convective term $\vrho_\veps \ue \otimes \ue$, taking advantage of the analysis presented in Subsection \ref{ss:wave_system} and performing equalities \eqref{eq:conv_term_rel} backwards, we find that 
\begin{equation*}
\int_0^T \int_{\R^2}\vrho_\veps \ue\otimes \ue :\nabla \vec \psi \, \dxdt \longrightarrow \int_0^T\int_{\R^2} \vu \otimes \vu :\nabla \vec \psi \, \dxdt
\end{equation*}
for $\veps \rightarrow 0$ and for all smooth divergence-free test functions $\vec \psi$. 

Let us consider now the Coriolis term. We can write:
\begin{equation*}
\int_0^T\int_{\R^2}\frac{1}{\veps}\vrho_\veps \ue^\perp \cdot \vec \psi\, \dxdt=\int_0^T\int_{\R^2}R_\veps \ue^\perp \cdot \vec \psi\, \dxdt+\int_0^T\int_{\R^2}\frac{1}{\veps} \ue^\perp \cdot \vec \psi\, \dxdt\, .
\end{equation*}
Since $\ue$ is divergence-free, the latter term vanishes when tested against such $\vec \psi$ defined in \eqref{def:test_function}. On the other hand, again thanks to Lemma \ref{l:reg_Ru_veps}, one can get 
 \begin{equation*}
 \int_0^T\int_{\R^2}R_\veps \ue^\perp \cdot \vec \psi\, \dxdt\longrightarrow \int_0^T\int_{\R^2}R \vu^\perp \cdot \vec \psi\, \dxdt\, .
 \end{equation*}
In the end, letting $\veps \rightarrow 0$ in \eqref{weak_mom_limit}, we gather   
\begin{equation*}
\int_0^T\!\!\!\int_{\R^2}  
	\left( -  \vu \cdot \partial_t \vec\psi -  \vu \otimes\vu  : \nabla_x \vec\psi 
	+  \,  R \vu^\perp \cdot \vec\psi  \right) \dxdt 
	= \int_{\R^2} \vu_0  \cdot \vec\psi (0,\cdot) \dx
\end{equation*}
for any test function $\vec \psi$ defined as in \eqref{def:test_function}. 

From this relation, we immediately obtain that 
\begin{equation*}
\d_t \vu +\div (\vu \otimes \vu) +R\vu^\perp +\nabla \Pi=0
\end{equation*}
for a suitable pressure term $\nabla \Pi$. This term appears as a result of the weak formulation of the problem. It can be viewed as a Lagrangian multiplier associated to the to the divergence-free constraint on $\vu$. Finally, the quantity $R$ satisfies the transport equation found in \eqref{eq:limit_transp_R}. 

We conclude this paragraph, devoting our attention to the analysis of the regularity of $\nabla \Pi$. We apply the $\div$ operator to the momentum equation in \eqref{QH-Euler_syst}, deducing that $\Pi$  satisfies 
\begin{equation}\label{time_reg_pressure}
-\Delta \Pi= \div G \qquad \text{where}\qquad G:= \vec u \cdot \nabla \vec u+R\vec u^\perp\, .
\end{equation}
On the one hand, 
Lemma 2 of \cite{D_JDE} gives 
\begin{equation*}
\|\nabla \Pi\|_{L^2}\leq C\|G\|_{L^2}\leq C\left(\|\vec u\|_{L^2}\|\nabla \vec u\|_{L^\infty}+\|R\|_{L^\infty}\|\vec u^\perp\|_{L^2}\right)\, .
\end{equation*}
This implies that $\nabla \Pi\in L^\infty_T(L^2)$.

On the other hand, owing to the divergence-free condition on $\vec u$, we have 
\begin{equation*}
\|\Delta \Pi\|_{H^{s-1}}\leq C\left(\|\vec u\|^2_{H^s}+\|R\|_{L^\infty}\|u\|_{H^s}+\|\nabla R\|_{H^{s-1}}\|\vec u\|_{L^\infty}\right)\, ,
\end{equation*}
where we have also used Proposition \ref{prop:app_fine_tame_est}.

In the end, we deduce that $\Delta \Pi \in L^\infty_T(H^{s-1})$. Thus, we conclude that $\nabla \Pi\in L^\infty_T(H^s)$.

At this point, employing classical results on solutions to transport equations in Sobolev spaces, we may infer the claimed $C^0$ time regularity of $\vec u$ and $R$. Moreover, thanks to the fact that $R$ and $\vec u$ are both continuous in time, from the elliptic equation \eqref{time_reg_pressure}, we get that also $\nabla \Pi \in C^0_T(H^s)$.

\section{Well-posedness for the quasi-homogeneous system}\label{s:well-posedness_Q-H}
In this section, for the reader's convenience, we review the well-posedness theory of the quasi-homogeneous Euler system \eqref{system_Q-H_thm}, in particular, the ``asymptotically global'' well-posedness result presented in \cite{C-F_sub}. In the first paragraph, we recall the local well-posedness theorem for system \eqref{system_Q-H_thm} in the $H^s$ framework. Actually, equations \eqref{system_Q-H_thm} are locally well-posedness in all $B^s_{p,r}$ Besov spaces, under the condition \eqref{Lip_assumption}. We refer to \cite{C-F_RWA} where the authors apply the standard Littlewood-Paley machinery to the quasi-homogeneous ideal MHD system to recover local in time well-posedness in spaces $B^s_{p,r}$ for any $1<p<+\infty$. The case $p=+\infty$ was reached in \cite{C-F_sub} with a different approach based on the vorticity formulation of the momentum equation.

In Subsection \ref{ss:improved_lifespan}, we explicitly derive the lower bound for the lifespan of solutions to system \eqref{system_Q-H_thm}. The reason in detailing the derivation of \eqref{T-ast:improved} for $T^\ast$ is due to the fact that it is much simpler than the one presented in \cite{C-F_sub}, where (due to the presence of the magnetic field) the lifespan behaves like the fifth iterated logarithm of the norms of the initial oscillation $R_0$ and the initial magnetic field. In addition, that lower bound (see \eqref{T-ast:improved} below) improves the standard lower bound coming from the hyperbolic theory, where the lifespan is bounded from below by the inverse of the norm of the initial data. 


\subsection{Local well-posedness in $H^s$ spaces}
In this subsection, we state the well-posedness result for system \eqref{system_Q-H_thm} in the $H^s$ functional framework with $s>2$, in which we have analysed the well-posedness issue for system \eqref{Euler_eps}. 
\begin{theorem}
Take $s>2$. Let $\big(R_0,u_0 \big)$ be initial data such that $R_0\in L^{\infty}$, with $\nabla R_0\in H^{s-1}$, and
the divergence-free vector filed $\vu_0 \,\in  H^s$.

Then, there exists a time $T^\ast > 0$ such that, on $[0,T^\ast]\times\R^2$, problem \eqref{system_Q-H_thm} has a unique solution $(R,\vu, \nabla \Pi)$ with:
\begin{itemize}
 \item $R\in C^0\big([0,T^\ast]\times \R^2\big)$ and $\nabla R\in C^0_{T^\ast}(H^{s-1}(\R^2))$;
 \item $\vu$ and $\nabla \Pi$ belong to $C^0_{T^\ast}( H^{s}(\R^2))$.
 \end{itemize}
In addition, if $T^\ast<T<+\infty$ and we assume that 
\begin{equation*}
\int_0^{T}  \big\| \nabla \vu(t) \big\|_{L^\infty}   \dt < +\infty\,, 
\end{equation*}
then the triplet $(R, \vu, \nabla \Pi)$ can be continued beyond $T$ into a solution of system \eqref{system_Q-H_thm} with the same regularity.
\end{theorem}

\subsection{Well-posedness in Besov spaces} \label{ss:W-P_Besov}
The main goal of this subsection is to review the lifespan estimate presented in \cite{C-F_sub} (for the MHD system) in order to get \eqref{improved_low_bound}. To show that, one has to work in critical Besov spaces where one can take advantage of the improved estimates for linear transport equations \textit{à la} Hmidi-Keraani-Vishik. In order to ensure that the condition \eqref{Lip_assumption} is satisfied, the lowest regularity space we can reach is $B^1_{\infty,1}$. In addition, since $\vec u\in B^1_{\infty,1}$, we have that the $B^0_{\infty,1}$ norm of the $\curl \vu$ can be bounded \textit{linearly} with respect to $\|\nabla \vu\|_{L^1_t(L^\infty)}$, instead of exponentially as in classical $B^s_{p,r}$ estimates (see Theorem \ref{thm:improved_est_transport} in the Appendix).

Finally, we construct a ``bridge'' between $H^s$ and $B^1_{\infty,1}$ Besov spaces establishing a continuation criterion, in the spirit of the one by Beale-Kato-Majda in \cite{B-K-M}.

We start by recalling the local well-posedness result for system \eqref{QH-Euler_syst} in $B^s_{\infty,r}$ and, in particular, in the end-point space $B^1_{\infty,1}$. We highlight that the physically relevant $L^2$ condition on $\vec u$, in the following theorem, is necessary to control the low frequency part of the solution, so as to reconstruct the velocity from its $\curl$ (see Lemma \ref{l:rel_curl} below).
\begin{theorem}\label{thm:W-P_besov_spaces_p-infty}
Let $(s,r)\in \R\times [1,+\infty]$ such that $s>1$ or $s=r=1$. Let $(R_0, \vec u_0)$ be an initial datum such that $R_0\in B^s_{\infty, r}(\R^2)$ and the divergence-free vector field $\vec u_0\in L^2(\R^2)\cap B^s_{\infty,r}(\R^2)$. Then, there exists a time $T^\ast>0$ such that system \eqref{QH-Euler_syst} has a unique solution $(R, \vec u)$ with the following regularity properties, if $r<+\infty$:
\begin{itemize}
\item $R\in C^0([0,T^\ast];B^s_{\infty,r}(\R^2))\cap C^1([0,T^\ast];B^{s-1}_{\infty,r}(\R^2)) $;
\item $\vec u$ and $\nabla \Pi$ belong to $C^0([0,T^\ast];B^s_{\infty,r}(\R^2))\cap C^1([0,T^\ast];L^2(\R^2)\cap B^{s-1}_{\infty,r}(\R^2))$. 
\end{itemize}
 In the case when $r=+\infty$, we need to replace $C^0([0,T^\ast];B^s_{\infty,r}(\R^2))$ by the space $C_w^0([0,T^\ast];B^s_{\infty,r}(\R^2))$.
\end{theorem}

Next, one can state the following continuation criterion for solutions of system \eqref{QH-Euler_syst} in $B^s_{\infty,r}$, where the couple $(s,r)$ satisfies the Lipschitz condition \eqref{Lip_assumption} (see \cite{C-F_sub} for details of the proof).
\begin{proposition}\label{prop:cont_criterion_Bes}
Let $(R_0,\vu_0)\in B^s_{\infty,r}\times (L^2\cap B^s_{\infty,r})$ with $\div \vu_0=0$. Given a time $T>0$, let $(R,\vu)$ be a solution of \eqref{QH-Euler_syst} on $[0,T[$ that belongs to $L^\infty_t(B^s_{\infty,r})\times L^\infty_t(L^2\cap B^s_{\infty,r})$ for any $t\in [0,T[$. If we assume that  
\begin{equation}\label{cont_cond_Bes}
\int_0^T\|\nabla \vu \|_{L^\infty}\, dt<+\infty\, ,
\end{equation} 
then $(R, \vu)$ can be continued beyond $T$ into a solution of \eqref{QH-Euler_syst} with the same regularity.


Moreover, the lifespan of a solution $(R,\vu)$ to system \eqref{system_Q-H_thm} does not depend on $(s,r)$ and, in particular, the lifespan of solutions in Theorem \ref{thm:well-posedness_Q-H-Euler} is the same as the lifespan in
$B^1_{\infty,1}\times \left( L^2 \cap B^1_{\infty,1}\right)$.
\end{proposition}

\subsection{The asymptotically global well-posedness result}\label{ss:improved_lifespan}
In this paragraph 
we focus on finding an asymptotic behaviour (in the regime of small oscillations for the densities) for the lifespan of solutions to system \eqref{QH-Euler_syst}. Namely, for small fluctuations $R_0$ of size $\delta>0$, the lifespan of solutions to this system tends to infinity when $\delta\rightarrow 0$. To show that, we have to take advantage of the \textit{linear} estimate in Theorem \ref{thm:improved_est_transport} for the transport equations in Besov spaces with zero regularity index. For that reason, it is important to work with the vorticity formulation of \eqref{QH-Euler_syst}, since $\omega \in B^0_{\infty,1}$. Thanks to the continuation criterion presented in Proposition \ref{prop:cont_criterion_Bes}, it is enough to find the bound of the lifespan in the lowest regularity space $B^1_{\infty,1}$. 


To begin with, we recall a general relation between a function and its $\curl$ that will be useful in the sequel. We refer to \cite{C-F_sub} for details of the proof.
\begin{lemma}\label{l:rel_curl}
Assume $f\in (L^2 \cap B^s_{\infty ,r})(\R^2)$ to be divergence-free. Denote by $\curl f:=-\d_2f_1+\d_1f_2$ its $\curl$ in $\R^2$.
Then, we have
\begin{equation}\label{eq:rel_curl}
\|f\|_{L^2\cap B^s_{\infty ,r}}\sim\|f\|_{L^2}+\|\curl f\|_{B^{s-1}_{\infty, r}}\, .
\end{equation}
\end{lemma}

Therefore, due to the relation \eqref{eq:rel_curl}, we can define
\begin{equation}\label{def_mathcal_E}
\mathcal{E}(t):= \|\vu(t)\|_{L^2}+\|\omega(t)\|_{B^0_{\infty,1}}\sim \|\vu(t)\|_{L^2\cap B^1_{\infty,1}}\, ,
\end{equation}
where we have set $\omega :=\curl \vu=-\d_2u_1+\d_1u_2$, as above.

Since the $L^2$ norm of the velocity field is preserved, to control $\vu$ in $B^1_{\infty ,1}$, it will be enough to find estimates for $\curl \vu$ in $B^{0}_{\infty ,1}$. Hence, we apply again the $\curl$ operator to the second equation in system \eqref{QH-Euler_syst} to get the following system 
\begin{equation}\label{QH-Euler_vorticity}
\begin{cases}
\d_tR+\vu \cdot \nabla R=0\\
\d_t \omega +\vu \cdot \nabla \omega=-\div (R\vu)\, .
\end{cases}
\end{equation}
where we have defined $\omega :=\curl \vu=-\d_2u_1+\d_1u_2$.

Making use of Theorem \ref{thm:improved_est_transport}, we obtain 
\begin{equation*}
\|\omega (t)\|_{B^0_{\infty,1}}\leq C \left( \|\omega (0)\|_{B^0_{\infty,1}}+\int_0^t\|\div (R\vu)\|_{B^0_{\infty,1}}\detau\right)\left(1+\int_0^t\|\nabla \vu\|_{L^\infty}\detau\right)\, .
\end{equation*}
Now, we look at the bound for $\div (R\vu)$, finding that 
\begin{equation*}
\|\div (R\vu)\|_{B^0_{\infty,1}}\leq C \left(\|R\|_{L^\infty}\|\vu\|_{B^1_{\infty,1}}+\|\vu\|_{L^\infty}\|R\|_{B^1_{\infty,1}}\right)\leq C \|R\|_{B^1_{\infty,1}}\mathcal{E}(t)\, .
\end{equation*}
Then, we deduce 
\begin{equation}\label{eq:en_est_1}
\mathcal{E}(t)\leq C \left(\mc E(0)+\int_0^t\mc E(\tau)\|R\|_{B^1_{\infty,1}}\, \detau \right)\left(1+\int_0^t\mc E(\tau)\, \detau\right)\, .
\end{equation}
At this point, Theorem \ref{thm_transport} implies that 
\begin{equation*}
\|R(t)\|_{B^1_{\infty,1}}\leq \|R_0\|_{B^1_{\infty,1}}\exp \left(C\int_0^t \mc E(\tau)\, \detau\right)\, .
\end{equation*}
Plugging this bound into \eqref{eq:en_est_1} gives
\begin{equation*}
\mc E(t)\leq C \left(1+\int_0^t\mc E(\tau)\, \detau\right)\left(\mc E(0)+\|R_0\|_{B^1_{\infty,1}}\int_0^t \mc E(\tau)\exp \left(\int_0^\tau \mc E(s)\, \ds \right)\, \detau \right)\, .
\end{equation*}
We define
\begin{equation}\label{def_T^ast}
T^\ast:=\sup \left\{t>0:\|R_0\|_{B^1_{\infty,1}}\int_0^t \mc E(\tau)\exp \left(\int_0^\tau \mc E(s)\, \ds \right)\, \detau \leq \mc E(0)\right\}\, .
\end{equation}
Then, for all $t\in [0,T^\ast]$, we deduce 
\begin{equation*}
\mc E(t)\leq C\left(1+\int_0^t\mc E(\tau)\, \detau\right)\mc E(0)
\end{equation*}
and thanks to the Gr\"onwall's lemma we infer 
\begin{equation}\label{eq:en_est_2}
\mc E(t)\leq C\mc E(0)\,  e^{C\mc E(0)t}\, ,
\end{equation}
for a suitable constant $C>0$.

It remains to find a control on the integral of $\mc E(t)$. We have 
\begin{equation*}
\int_0^t\mc E(\tau) \, \detau \leq e^{C\mc E(0)t}-1
\end{equation*}
and, due to the previous bound \eqref{eq:en_est_2}, we get
\begin{equation}\label{estimate_integral_energy}
\begin{split}
\|R_0\|_{B^1_{\infty,1}}\int_0^t \mc E(\tau)\exp \left(\int_0^\tau \mc E(s)\, \ds \right)\, \detau &\leq C\|R_0\|_{B^1_{\infty,1}}\int_0^t \mc E(0)\, e^{C\mc E(0)\tau}\exp \left(e^{C\mc E(0)\tau}-1 \right)\, \detau \\
&\leq C\|R_0\|_{B^1_{\infty,1}}\left(\exp \left(e^{C\mc E(0)t}-1\right)-1\right)\, .
\end{split}
\end{equation}
Finally, by definition \eqref{def_T^ast} of $T^\ast$, we can argue  
\begin{equation*}
\mc E(0)\leq C\|R_0\|_{B^1_{\infty,1}}\left(\exp \left(e^{C\mc E(0)T^\ast}-1\right)-1\right)\, ,
\end{equation*}
that give the following lower bound for the lifespan of solutions:
\begin{equation*}
T^\ast\geq \frac{C}{\mc E(0)}\log\left(\log \left(C\frac{\mc E(0)}{\|R_0\|_{B^1_{\infty,1}}}+1\right)+1\right)\, .
\end{equation*}
From there, recalling the definition \eqref{def_mathcal_E} for $\mc E(0)$, we have 
\begin{equation}\label{T-ast:improved}
T^\ast\geq \frac{C}{\|\vu_0\|_{L^2\cap B^1_{\infty,1}}}\log\left(\log \left(C\frac{\|\vu_0\|_{L^2\cap B^1_{\infty,1}}}{\|R_0\|_{B^1_{\infty,1}}}+1\right)+1\right)\, ,
\end{equation}
for a suitable constant $C>0$. This is the claimed lower bound stated in Theorem \ref{thm:well-posedness_Q-H-Euler}.

\section{The lifespan of solutions to the primitive problem}\label{s:lifespan_full}
The main goal of this section is to present an ``asymptotically global'' well-posedness result for system \eqref{Euler-a_eps_1}, when the size of fluctuations of the densities goes to zero, in the spirit of Subsection \ref{ss:improved_lifespan}. We start by showing a continuation type criterion for system \eqref{Euler-a_eps_1} and discussing the related consequences (see Subsection \ref{ss:cont_criterion+consequences} below for details).
We conclude this section presenting the asymptotic behaviour of the lifespan of solutions to system \eqref{Euler-a_eps_1}: the lifespan may be very large, if the size of non-homogeneities $a_{0,\veps}$ defined in \eqref{def_a_veps} is small (see relation \eqref{asymptotic_time} below). We point out that it is \textit{not} clear at all that the global existence still holds even in a fast rotation regime. 
 
\subsection{The continuation criterion and consequences}\label{ss:cont_criterion+consequences}
In this paragraph, we start by presenting a continuation type result in Sobolev spaces for system \eqref{Euler-a_eps_1}, in the spirit of the Beale-Kato-Majda continuation criterion \cite{B-K-M}. 
We refer to the work \cite{B-L-S} by Bae, Lee and Shin, regarding the details of the proof. 
\begin{proposition}\label{prop:cont_criterion_original_prob}
Take $\veps \in\;  ]0,1]$ fixed. Let $(a_{0,\veps},\vu_{0,\veps})\in L^\infty \times H^s$ with $\nabla a_{0,\veps}\in H^{s-1} $ and $\div \vu_{0,\veps}=0$. Given a time $T>0$, let $(a_\veps,\vu_\veps, \nabla \Pi_\veps)$ be a solution of \eqref{Euler-a_eps_1} on $[0,T[$ that belongs to $L^\infty_t(L^\infty)\times L^\infty_t(H^s)\times L^\infty_t(H^s)$ and $\nabla a_\veps \in L^\infty_t(H^{s-1}) $ for any $t\in [0,T[$. If we assume that  
\begin{equation}\label{cont_cond_orig_prob}
\int_0^T\|\nabla \vu_\veps \|_{L^\infty}\, \dt<+\infty\, ,
\end{equation} 
then $(a_\veps, \vu_\veps,\nabla \Pi_\veps)$ can be continued beyond $T$ into a solution of \eqref{Euler-a_eps_1} with the same regularity.
\end{proposition}

At this point, 
if we are able to control the norm $\|\ue \|_{L^\infty_T(L^2\cap B^1_{\infty,1})}$, then we are able to bound $\|\nabla \ue \|_{L^\infty_T(L^\infty)}$. This will imply \eqref{cont_cond_orig_prob} and, therefore, the solution will exist until the time $T$. 


Indeed, we have that 
\begin{equation*}
\|\nabla \ue \|_{L^\infty_T(L^\infty)}\leq C\, \| \ue \|_{L^\infty_T(B^1_{\infty,1})}\, .
\end{equation*} 
As already pointed out in Lemma \ref{l:rel_curl}, to control the $B^1_{\infty,1}$ of $\ue$ it is enough to have a $L^2$ estimate for $\ue$ and a $B^0_{\infty, 1}$ estimate for its $\curl$. Those estimates are the topic of the next Subsection \ref{ss:asym_lifespan}, provided that the time $T>0$ is defined as in \eqref{def_T^ast_veps} below. 
Therefore, $\|\nabla \ue\|_{L^\infty_T(L^\infty)}<+\infty$ and so 
$$ \int_0^T \|\nabla \ue\|_{L^\infty}<+\infty\, ,$$
for such $T>0$. 

Finally, we note that we have already shown the existence and uniqueness of solutions to system \eqref{Euler-a_eps_1} in the Sobolev spaces $H^s$ with $s>2$ (see Section \ref{s:well-posedness_original_problem}) and thanks to Proposition \ref{p:embed}, those spaces are continuously embedded in the space $B^1_{\infty,1}$.

\subsection{The asymptotic lifespan}\label{ss:asym_lifespan}
In this paragraph we focus our attention on the lifespan of solutions to the primitive system \eqref{Euler-a_eps_1}. As we will see in the sequel, the fast rotation effects are not enough to get an ``asymptotically global'' well-posedness result. Similarly to the quasi-homogeneous case, we need smallness for the $B^1_{\infty,1}$ norm of the initial fluctuations $a_{0,\veps}$ defined in \eqref{def_a_veps}.

First of all, we have to take advantage of the vorticity formulation of system \eqref{Euler-a_eps_1}. To do so, we apply the $\curl$ operator to the momentum equation, obtaining 
\begin{equation}\label{curl_eq}
\d_t  \omega_\veps +\ue \cdot \nabla  \omega_\veps+ \nabla a_\veps\,  \wedge\,  \nabla \Pi_\veps=0\, ,
\end{equation}
where $\omega_\veps:=\curl \ue$ and $\nabla a_\veps\,  \wedge\,  \nabla \Pi_\veps:=\d_1 a_\veps \,\d_2\Pi_\veps-\d_2 a_\veps \d_1\Pi_\veps$.

We notice that vorticity formulation is the key point to bypass the issues coming from the Coriolis force, which singular effects disappear in \eqref{curl_eq}.

Next, we make use of Theorem \ref{thm:improved_est_transport} (in the Appendix) and we deduce that 
\begin{equation}\label{transp_curl}
\| \omega_\veps \|_{B^0_{\infty,1}}\leq C \left(\| \omega_{0,\veps}\|_{B^0_{\infty,1}}+\int_0^t \|\nabla a_\veps\, \wedge\, \nabla \Pi_\veps \|_{B^0_{\infty,1}}\detau \right)\left(1+\int_0^t\|\nabla \ue\|_{L^\infty} \detau\right)\, .
\end{equation}
We start by bounding the $B^0_{\infty,1}$ norm of $\nabla a_\veps\, \wedge\, \nabla \Pi_\veps$. We observe that 
\begin{equation}\label{l-p_dec_wedge}
\begin{split}
\d_1a_\veps\, \d_2 \Pi_\veps-\d_2a_\veps\, \d_1\Pi_\veps&=\mc T_{\d_1a_\veps}\d_2 \Pi_\veps-\mc T_{\d_2a_\veps}\d_1 \Pi_\veps+\mc T_{\d_2\Pi_\veps}\d_1 a_\veps-\mc T_{\d_1\Pi_\veps}\d_2 a_\veps\\
&+\d_1\mc R(a_\veps -\Delta_{-1}a_\veps,\, \d_2 \Pi_\veps)-\d_2\mc R(a_\veps -\Delta_{-1}a_\veps,\, \d_1 \Pi_\veps)\\
&+\mc R(\d_1\Delta_{-1}a_\veps,\, \d_2 \Pi_\veps)+\mc R(\d_2\Delta_{-1}a_\veps,\, \d_1 \Pi_\veps)\, .
\end{split}
\end{equation}


Applying Proposition \ref{T-R} directly to the terms involving the paraproduct $\mc T$, we have 
\begin{equation*}
\|\mc T_{\nabla a_\veps}\nabla \Pi_\veps\|_{B^0_{\infty,1}}+\|\mc T_{\nabla \Pi_\veps}\nabla  a_\veps\|_{B^0_{\infty,1}}\leq C \left(\|\nabla a_\veps\|_{L^\infty}\|\nabla \Pi_\veps\|_{B^0_{\infty,1}}+\|\nabla a_\veps\|_{B^0_{\infty,1}}\|\nabla \Pi_\veps\|_{L^\infty}\right)\, .
\end{equation*}
Next, we have to deal with the remainders $\mc R$. We start by bounding the $B^0_{\infty,1}$ norm of $\d_1\mc R(a_\veps -\Delta_{-1}a_\veps,\, \d_2 \Pi_\veps)$. 
One has:
\begin{equation*}
\begin{split}
\|\d_1\mc R(a_\veps -\Delta_{-1}a_\veps,\, \d_2 \Pi_\veps)\|_{B^0_{\infty,1}}&\leq C\|\mc R(a_\veps -\Delta_{-1}a_\veps,\, \d_2 \Pi_\veps)\|_{B^1_{\infty,1}}\\
&\leq C\left(\|\nabla \Pi_\veps\|_{B^0_{\infty,\infty}}\|(\text{Id}-\Delta_{-1})\, a_\veps\|_{B^1_{\infty,1}}\right)\\
&\leq C \left(\|\nabla \Pi_\veps\|_{L^\infty}\|\nabla a_\veps\|_{B^0_{\infty,1}}\right)\, ,
\end{split}
\end{equation*}
where we have employed the localization properties of the Littlewood-Paley decomposition. In a similar way, one can argue for 
$\d_2\mc R(a_\veps -\Delta_{-1}a_\veps,\, \d_1 \Pi_\veps)$.

It remains to bound $\mc R(\d_1\Delta_{-1}a_\veps,\, \d_2 \Pi_\veps)$. Analogously, one can treat the term $\mc R(\d_2\Delta_{-1}a_\veps,\, \d_1 \Pi_\veps)$ in \eqref{l-p_dec_wedge}. 
We obtain that 
\begin{equation*}
\|\mc R(\d_1\Delta_{-1}a_\veps,\, \d_2 \Pi_\veps)\|_{B^0_{\infty,1}}\leq C \|\mc R(\d_1\Delta_{-1}a_\veps,\, \d_2 \Pi_\veps)\|_{B^1_{\infty,1}}\leq C\left(\|\nabla \Pi_\veps\|_{L^\infty}\|\d_1\Delta_{-1}a_\veps\|_{B^1_{\infty,1}}\right)\, .
\end{equation*}
Employing the spectral properties of operator $\Delta_{-1}$, one has that 
\begin{equation*}
\|\d_1\Delta_{-1}a_\veps\|_{B^1_{\infty,1}}\leq C\|\Delta_{-1}\nabla a_\veps\|_{L^\infty}\, .
\end{equation*}
Then, 
$$ \|\mc R(\d_1\Delta_{-1}a_\veps,\, \d_2 \Pi_\veps)\|_{B^0_{\infty,1}}\leq C \left(\|\nabla \Pi_\veps\|_{L^\infty}\|\nabla a_\veps\|_{B^0_{\infty,1}}\right)\, .$$
Finally, we get
\begin{equation*}
\|\nabla a_\veps\, \wedge\, \nabla \Pi_\veps \|_{B^0_{\infty,1}}\leq C\left(\|\nabla a_\veps\|_{L^\infty}\|\nabla \Pi_\veps\|_{B^0_{\infty,1}}+\|\nabla a_\veps\|_{B^0_{\infty,1}}\|\nabla \Pi_\veps\|_{L^\infty}\right)\, .
\end{equation*}
So plugging the previous estimate in \eqref{transp_curl}, one gets
\begin{equation}
\| \omega_\veps \|_{B^0_{\infty,1}}\leq C \left(\| \omega_{0,\veps}\|_{B^0_{\infty,1}}+\int_0^t \|\nabla a_\veps\|_{B^0_{\infty,1}} \| \nabla \Pi_\veps \|_{B^0_{\infty,1}}\detau \right)\left(1+\int_0^t\|\nabla \ue\|_{L^\infty} \detau\right)\, .
\end{equation}
At this point, we define
\begin{equation}\label{def_energy_and_A}
E_\veps(t):=\|\ue (t)\|_{L^2\cap B^1_{\infty,1}}\qquad \text{and}\qquad \mc A_\veps(t):=\|\nabla a_\veps (t)\|_{B^0_{\infty,1}}\, .
\end{equation}
In this way, we have 
\begin{equation}\label{energy}
E_\veps(t)\leq C \left(E_\veps (0)+\int_0^t \mc A_\veps(\tau) \|\nabla \Pi_\veps\|_{B^0_{\infty,1}}\detau \right)\left(1+\int_0^t E_\veps (\tau )\detau \right)\, .
\end{equation}
Next, we recall that, for $i=1,2$:
\begin{equation*}
\d_t\, \d_i a_\veps+\ue \cdot \nabla \, \d_i a_\veps=-\d_i \ue \cdot \nabla a_\veps
\end{equation*}
and, due to the divergence-free condition on $\vec u_\veps$, we can write
\begin{equation*}
\d_i \ue \cdot \nabla a_\veps=\sum_j\d_i \ue^j \, \d_ja_\veps=\sum_j\Big(\d_i(\ue^j\, \d_j a_\veps)-\d_j(\ue^j\, \d_ia_\veps)\Big)\, .
\end{equation*}
So, using Proposition \ref{T-R} and the fact that 
\begin{equation*}
\|\mc \d_i R(\ue^j,\, \d_j a_\veps)\|_{B^0_{\infty,1}}\leq C\, \|\mc R(\ue^j,\, \d_j a_\veps)\|_{B^1_{\infty,1}}\leq C\, \|\nabla a_\veps\|_{B^0_{\infty,1}}\|\ue\|_{B^1_{\infty,1}}\, ,
\end{equation*}
we may finally get
\begin{equation*}
\|\d_i \ue \cdot \nabla a_\veps\|_{B^0_{\infty,1}}\leq C \|\nabla a_\veps\|_{B^0_{\infty,1}}\|\ue \|_{B^1_{\infty,1}}\, .
\end{equation*}
Thus, 
\begin{equation*}
\|\nabla a_\veps (t)\|_{B^0_{\infty,1}}\leq \|\nabla a_{0, \veps}\|_{B^0_{\infty,1}}\exp \left(C\int_0^t \| \ue\|_{B^1_{\infty,1}}\detau \right)\, .
\end{equation*}
Therefore, recalling \eqref{def_energy_and_A}, one has
\begin{equation}\label{est_A}
\mc A_\veps(t)\leq \mc A_\veps (0) \exp \left( C\int_0^t E(\tau) \detau\right)\, .
\end{equation}
The next goal is to bound the pressure term in $B^0_{\infty,1}$. Actually, we shall bound its $B^1_{\infty,1}$ norm. Similarly to the analysis performed in Subsection \ref{ss:unif_est} for the $H^s$ norm (see e.g. inequality \eqref{est_Pi_H^s_1}), there exists some exponent $\lambda \geq 1$ such that
\begin{equation}\label{est_pres}
\|\nabla \Pi_\veps\|_{B^1_{\infty,1}}\leq C\left(\left(1+\veps \|\nabla a_\veps \|_{B^0_{\infty,1}}^\lambda \right)\|\nabla \Pi_\veps \|_{L^2}+\veps\, \|\vrho_\veps\, \div (\ue \cdot \nabla \ue)\|_{B^0_{\infty,1}}+\|\vrho_\veps \, \div \ue^\perp\|_{B^0_{\infty,1}}\right)\, .
\end{equation}
The $L^2$ estimate for the pressure term follows in a similar way to one performed in \eqref{est:Pi_L^2}, i.e. 
\begin{equation}\label{L^2-est-pressure}
\|\nabla \Pi_\veps\|_{L^2}\leq C\, \veps \, \|\ue \|_{L^2}\|\nabla \ue \|_{L^\infty}+\|\ue^\perp \|_{L^2}\, .
\end{equation}

Next, as showed above in the bound for $\|\d_i \ue \cdot \nabla a_\veps\|_{B^0_{\infty,1}}$, combining Bony's decomposition with the fact that $\div (\ue \cdot \nabla \ue)=\nabla \ue :\nabla \ue$, we may infer:
\begin{equation*}
\|\div (\ue \cdot \nabla \ue)\|_{B^0_{\infty,1}}\leq C\, \|\ue \|^2_{B^1_{\infty,1}}\, .
\end{equation*} 
Now, from Proposition 3 of \cite{D-F_JMPA}, we can estimate the $B^1_{\infty,1}$ norm of the density in the following way:
\begin{equation*}
\|\vrho_\veps\|_{B^1_{\infty,1}}\leq C\, \left( \oline\vrho+\veps\, \|\nabla a_\veps\|_{B^0_{\infty,1}}\right)\, .
\end{equation*}
Finally, plugging the $L^2$ estimate \eqref{L^2-est-pressure} and all the above inequalities in \eqref{est_pres}, one may conclude that
\begin{equation}\label{est_pres_1}
\begin{split}
\|\nabla \Pi_\veps\|_{B^1_{\infty,1}}&\leq C \left(1+\veps\, \|\nabla a_\veps \|_{B^0_{\infty,1}}^\lambda \right)\left(\veps\, \|\ue\|_{L^2}\|\nabla \ue\|_{B^0_{\infty,1}}+\|\ue \|_{L^2}\right)\\
&+C \left(1+\veps\, \|\nabla a_\veps \|_{B^0_{\infty,1}} \right)\left(\veps\, \| \ue\|_{B^1_{\infty,1}}^2+\|\ue \|_{B^1_{\infty,1}}\right)\\
&\leq C (1+\veps\, \mc A_\veps^\lambda)(\veps\, E_\veps^2+E_\veps)+C(1+\veps \mc A_\veps)(\veps\, E_\veps^2+E_\veps)\\
&\leq C (\veps \, E_\veps^2+E_\veps)(1+\veps\,\mc A_\veps +\veps\, \mc A_\veps^\lambda)\\
&\leq C (\veps \, E_\veps^2+E_\veps)(1 +\veps\, \mc A_\veps^\lambda)\, .
\end{split}
\end{equation}
We insert now in \eqref{energy} the estimates found in \eqref{est_pres_1} and in \eqref{est_A}, deducing that 
\begin{equation}\label{est_E}
E_\veps(t)\leq C\left(E_\veps(0)+\mc B_\veps(0)\int_0^t \exp \left(C \int_0^\tau E_\veps (s)\, \ds\right) \left(\veps\, E_\veps^2(\tau)+ E_\veps(\tau)\right)\, \detau \right)\left(1+\int_0^t E_\veps (\tau) \detau \right),
\end{equation}
where we have set $\mc B_\veps (0):=\mc A_\veps(0)+\veps\, \mc A_\veps(0)^{\lambda +1} $.

At this point, we define $T_\veps^\ast>0$ such that
\begin{equation}\label{def_T^ast_veps}
T_\veps^\ast:= \sup \left\{t>0 : \, \mc B_\veps(0)\int_0^t \exp \left(C \int_0^\tau E_\veps (s)\, \ds\right) \left(\veps\, E_\veps^2(\tau)+ E_\veps(\tau)\right)\, \detau \leq E_\veps(0) \right\}\, .
\end{equation}
So, from \eqref{est_E} and using Gr\"onwall's inequality, we obtain that
\begin{equation*}
E_\veps(t)\leq C\,  E_\veps(0)e^{CtE_\veps(0)}\, ,
\end{equation*}
for all $t\in [0,T_\veps^\ast]$.

The previous estimate implies that, for all $t\in [0,T_\veps^\ast]$, one has 
\begin{equation*}
\int_0^t E_\veps(\tau) \detau \leq  e^{CtE_\veps(0)}-1\, .
\end{equation*}
Analogously to inequality \eqref{estimate_integral_energy} in Subsection \ref{ss:improved_lifespan}, we can argue that  
\begin{equation*}
\mc B_\veps(0)\int_0^t \exp \left(C \int_0^\tau E_\veps (s)\, \ds\right)  E_\veps(\tau)\, \detau\leq C \mc B_\veps(0) \left(\exp \left(e^{CtE_\veps(0)}-1\right)-1\right)\, .
\end{equation*}
Then, it remains to control 
$$ \veps \mc B_\veps(0) \int_0^t \exp \left(C \int_0^\tau E_\veps (s)\, \ds\right)  E_\veps^2(\tau)\, \detau\, . $$ 
For this term, we may infer that 
\begin{equation*}
\begin{split}
\veps \mc B_\veps(0)\int_0^t \exp \left(C \int_0^\tau E_\veps (s)\, \ds\right)  E_\veps^2(\tau)\, \detau &\leq C\, \veps \,\mc B_\veps(0) \int_0^t E_\veps^2(0)e^{CtE_\veps(0)}\, \exp\left(e^{CtE_\veps(0)}-1\right)\, \detau\\
&\leq C\, \veps\, \mc B_\veps(0) E_\veps (0) \left(\exp \left(e^{CtE_\veps(0)}-1\right)-1\right)\, .
\end{split}
\end{equation*}
In the end, by definition \eqref{def_T^ast_veps} of $T^\ast_\veps$, we deduce
\begin{equation}\label{asymptotic_time}
T^\ast_\veps\geq \frac{C}{E_\veps(0)}\log\left(\log\left(\frac{CE_\veps(0)}{\max \{\mc B_\veps(0),\, \veps \, \mc B_\veps(0)E_\veps(0)\}}+1\right)+1\right),
\end{equation} 
for a suitable constant $C>0$. This concludes the proof of Theorem \ref{W-P_fullE}.

\appendix

\section{Appendix - A few tools from Fourier analysis} \label{app:FA}

This appendix is devoted to present some tools from Fourier analysis, which we have exploited in our analysis. We recall some definitions and properties of Littlewood-Paley theory, Besov spaces and paradifferential calculus. Moreover, using those notions, we will focus on the study of transport equations in Besov spaces.

\subsection{Littlewood-Paley theory and Besov spaces} \label{ss:LP_theory} 
In this paragraph, we recall the main results concerning the Littlewood-Paley theory. We refer e.g. to Chapter 2 of \cite{B-C-D} for details.
For simplicity of exposition, let us deal with the $\R^d$ case, with $d\geq1$; however, the whole construction can be adapted also to the $d$-dimensional torus $\TT^d$, and to the ``hybrid'' case
$\R^{d_1}\times\TT^{d_2}$.

\medbreak
First of all, let us introduce the \emph{Littlewood-Paley decomposition}. For this
we fix a smooth radial function $\chi$ such that $\Supp\chi\subset B(0,2)$, $\chi\equiv 1$ in a neighborhood of $B(0,1)$
and the map $r\mapsto\chi(r\,e)$ is non-increasing over $\R_+$ for all unitary vectors $e\in\R^d$.
Set $\varphi\left(\xi\right)=\chi\left(\xi\right)-\chi\left(2\xi\right)$ and $\vphi_j(\xi):=\vphi(2^{-j}\xi)$ for all $j\geq0$.
The dyadic blocks $(\Delta_j)_{j\in\Z}$ are defined by\footnote{We agree  that  $f(D)$ stands for 
the pseudo-differential operator $u\mapsto\mc{F}^{-1}[f(\xi)\,\what u(\xi)]$.} 
$$
\Delta_j\,:=\,0\quad\mbox{ if }\; j\leq-2,\qquad\Delta_{-1}\,:=\,\chi(D)\qquad\mbox{ and }\qquad
\Delta_j\,:=\,\varphi(2^{-j}D)\quad \mbox{ if }\;  j\geq0\,.
$$
For any $j\geq0$ fixed, we  also introduce the \emph{low frequency cut-off operator}
\begin{equation} \label{eq:S_j}
S_j\,:=\,\chi(2^{-j}D)\,=\,\sum_{k\leq j-1}\Delta_{k}\,.
\end{equation}
Note that $S_j$ is a convolution operator. More precisely, after defining
$$
K_0\,:=\,\mc F^{-1}\chi\qquad\qquad\mbox{ and }\qquad\qquad K_j(x)\,:=\,\mathcal{F}^{-1}[\chi (2^{-j}\cdot)] (x) = 2^{jd}K_0(2^j x)\,,
$$
for all $j\in\N$ and all tempered distributions $u\in\mc S'$, we have that $S_ju\,=\,K_j\,*\,u$.
Thus the $L^1$ norm of $K_j$ is independent of $j\geq0$, hence $S_j$ maps continuously $L^p$ into itself, for any $1 \leq p \leq +\infty$.

The following classical property holds true: for any $u\in\mc{S}'$, then one has the equality $u=\sum_{j}\Delta_ju$ in the sense of $\mc{S}'$.
Let us also recall the so-called \emph{Bernstein inequalities}.
  \begin{lemma} \label{l:bern}
Let  $0<r<R$.   A constant $C$ exists so that, for any non-negative integer $k$, any couple $(p,q)$ 
in $[1,+\infty]^2$, with  $p\leq q$,  and any function $u\in L^p$,  we  have, for all $\lambda>0$,
$$
\displaylines{
{\Supp}\, \widehat u \subset   B(0,\lambda R)\quad
\Longrightarrow\quad
\|\nabla^k u\|_{L^q}\, \leq\,
 C^{k+1}\,\lambda^{k+d\left(\frac{1}{p}-\frac{1}{q}\right)}\,\|u\|_{L^p}\;;\cr
{\Supp}\, \widehat u \subset \{\xi\in\R^d\,:\, \lambda r\leq|\xi|\leq \lambda R\}
\quad\Longrightarrow\quad C^{-k-1}\,\lambda^k\|u\|_{L^p}\,
\leq\,
\|\nabla^k u\|_{L^p}\,
\leq\,
C^{k+1} \, \lambda^k\|u\|_{L^p}\,.
}$$
\end{lemma}   

By use of Littlewood-Paley decomposition, we can define the class of Besov spaces.
\begin{definition} \label{d:B}
  Let $s\in\R$ and $1\leq p,r\leq+\infty$. The \emph{non-homogeneous Besov space}
$B^{s}_{p,r}$ is defined as the subset of tempered distributions $u$ for which
$$
\|u\|_{B^{s}_{p,r}}\,:=\,
\left\|\left(2^{js}\,\|\Delta_ju\|_{L^p}\right)_{j\geq -1}\right\|_{\ell^r}\,<\,+\infty\,.
$$
\end{definition}
Besov spaces are interpolation spaces between Sobolev spaces. In fact, for any $k\in\N$ and~$p\in[1,+\infty]$
we have the chain of continuous embeddings $ B^k_{p,1}\hookrightarrow W^{k,p}\hookrightarrow B^k_{p,\infty}$,
which, when $1<p<+\infty$, can be refined to
$$
 B^k_{p, \min (p, 2)}\hookrightarrow W^{k,p}\hookrightarrow B^k_{p, \max(p, 2)}\,.
$$
In particular, for all $s\in\R$ we deduce that $B^s_{2,2}\equiv H^s$, with equivalence of norms:
\begin{equation} \label{eq:LP-Sob}
\|f\|_{H^s}\,\sim\,\left(\sum_{j\geq-1}2^{2 j s}\,\|\Delta_jf\|^2_{L^2}\right)^{\!\!1/2}\,.
\end{equation}

As an immediate consequence of the first Bernstein inequality, one gets the following embedding result, which generalises the Sobolev embeddings.
\begin{proposition}\label{p:embed}
The space $B^{s_1}_{p_1,r_1}$ is continuously embedded in the space $B^{s_2}_{p_2,r_2}$ for all indices satisfying $p_1\,\leq\,p_2$ and either
$s_2\,<\,s_1-d\big(1/p_1-1/p_2\big)$, or $s_2\,=\,s_1-d\big(1/p_1-1/p_2\big)$ and $r_1\leq r_2$.
\end{proposition}
In particular, we get the following chain of continuous embeddings:
$$ B^s_{p,r}\hookrightarrow B^{s-d/p}_{\infty,r}\hookrightarrow B^0_{\infty,1}\hookrightarrow L^\infty \, , $$
whenever the triplet $(s,p,r)\in \R\times [1,+\infty]^2$ satisfies 
\begin{equation}\label{cond_algebra}
s>\frac{d}{p} \quad \quad \quad \text{or}\quad \quad \quad s=\frac{d}{p} \quad \text{and}\quad r=1\, .
\end{equation}

\subsection{Paradifferential calculus}\label{app_paradiff}
Let us now introduce the Bony decomposition (see \cite{Bony}). Once again, we refer to Chapter 2 of \cite{B-C-D} for details. Formally, the product of two tempered distributions $u$ and $v$ can be decomposed into
$$ u\,v=\mathcal{T}_uv+\mathcal{T}_vu+\mathcal{R}(u,v)\, , $$
where the \textit{paraproduct} $\mathcal{T}$ and the \textit{remainder} $\mathcal{R}$ are defined as follows: 
$$ \mathcal{T}_uv=\sum_jS_{j-1}u\, \Delta_jv \quad \text{and}\quad  \mathcal{R} (u,v):= \sum_j\sum_{|k-j|}\Delta_ju\, \Delta_kv\, . $$
The paraproduct and remainder operators have nice continuity properties. The following ones will be of constant use in the paper.
\begin{proposition}\label{T-R}
For any $(s,p,r)\in \R \times [1, +\infty ]^2$ and $t>0$, the paraproduct operator $\mathcal{T}$ maps continuously $L^\infty \times B^s_{p,r}$ into $B^s_{p,r}$ and $B^{-t}_{\infty,\infty}\times B^s_{p,r}$ into $B^{s-t}_{p,r}$. Moreover, we have the following estimates
\begin{equation*}
\|\mathcal{T}_u v\|_{B^s_{p,r}}\leq C\|u\|_{L^\infty} \|\nabla v\|_{B^{s-1}_{p,r}} \quad \text{and}\quad \|\mathcal{T}_u v\|_{B^{s-t}_{p,r}}\leq C\|u\|_{B^{-t}_{\infty,\infty}} \|\nabla v\|_{B^{s-1}_{p,r}}\, .
\end{equation*} 
For any $(s_1,p_1,r_1)$ and $(s_2,p_2,r_2)$ in $\R \times [1, +\infty ]^2$ such that $s_1+s_2>0$, $1/p:=1/p_1+1/p_2\leq 1$ and $1/r:=1/r_1+1/r_2\leq 1$, the remainder operator $\mathcal{R}$ maps continuously $B^{s_1}_{p_1,r_1}\times B^{s_2}_{p_2,r_2}$ into $B^{s_1+s_2}_{p,r}$. In the case when $s_1+s_2=0$ and $1/r_1+1/r_2=1$, the operator $\mathcal{R}$ is continuous from $B^{s_1}_{p_1,r_1}\times B^{s_2}_{p_2,r_2}$ to $B^{0}_{p,\infty}$.  
\end{proposition}

As a consequence of the Proposition \ref{T-R}, the spaces $B^s_{p ,r}$ are Banach algebras provided condition \eqref{cond_algebra} holds for $s>0$. Moreover, in that case, we have the so-called \textit{tame estimates}. 

\begin{corollary}\label{cor:tame_est}
Let $(s,p,r)\in \R \times [1,+\infty ]^2$. Then, for every $f,g\in L^\infty \cap B^s_{p ,r}$, one has 
\begin{equation*}
\|fg\|_{B^s_{p ,r}}\leq C \left(\|f\|_{L^\infty}\|g\|_{B^s_{p ,r}}+\|f\|_{B^s_{p ,r}}\|g\|_{L^\infty}\right)\, .
\end{equation*}
\end{corollary}
\begin{remark}
The space $B^0_{\infty ,1}$ is \textit{not} an algebra. If $f,g \in B^0_{\infty ,1}$, applying Proposition \ref{T-R}, one can bound the paraproducts $\mathcal{T}_fg$ and $\mathcal{T}_gf$ but \textit{not} the remainder $\mathcal{R} (f,g)$.
\end{remark}

To end this paragraph, we present a fine estimate for products in which one of the two functions is only bounded but its gradient belongs to the Besov space $B^{s-1}_{p,r}$. 
\begin{proposition}\label{prop:app_fine_tame_est}
Let $(s,p,r)\in\; ]0,+\infty[\; \times \, [1,+\infty]^2$ satisfies condition \eqref{cond_algebra}. Assume that $g\in L^\infty \cap B^{s}_{p,r}$ and $f$ is a bounded function such that $\nabla f\in B^{s-1}_{p,r}$. Then, the product $fg$ belongs to $L^\infty \cap B^s_{p,r}$ and one has the following estimate:
\begin{equation*}
\|fg\|_{B^s_{p,r}}\leq C\left(\|f\|_{L^\infty}\|g\|_{B^s_{p,r}}+\|\nabla f\|_{B^{s-1}_{p,r}}\|g\|_{L^\infty}\right)\, .
\end{equation*}
\end{proposition}
\begin{proof}
Taking advantage of Bony decomposition, one can write 
\begin{equation*}
fg =\mathcal{T}_{f}g +\mathcal{T}_{g }f + \mathcal{R}(f, g )
\end{equation*}
and employing Proposition \ref{T-R}, we get
\begin{equation*}
\begin{split}
\|\mathcal{T}_{f}g\|_{B^s_{p,r}}&\leq C\|f\|_{L^\infty}\|g \|_{B^s_{p,r}}\\
\|\mathcal{T}_{g }\, f\|_{B^s_{p,r}}&\leq C\|g \|_{L^\infty}\|\nabla f\|_{B^{s-1}_{p,r}}\\
\|\mathcal{R}(f ,g)\|_{B^s_{p,r}}&\leq C\|f\|_{B^0_{\infty,\infty}}\|g \|_{B^s_{p,r}}\leq C\|f\|_{L^\infty}\|g  \|_{B^s_{p,r}}\, .
\end{split}
\end{equation*}
This completes the proof of the proposition.
\end{proof}

\subsection{Commutator estimates}\label{app:comm_est}
In this paragraph, we recall the main commutator estimates widely employed throughout the paper (we refer to Chapter 2 of \cite{B-C-D} for full details). 
\begin{definition}
We say that the triplet $(s,p,r)\in \R \times [1, +\infty ]^2$ satisfies the Lipschitz condition if 
\begin{equation} \label{Lip_cond}
s>1+d/p \quad \text{and}\quad r\in [1,+\infty ]\quad \quad \quad \text{or}\quad \quad \quad s=1+d/p \quad \text{and}\quad r=1\, .
\end{equation}
\end{definition}

\smallskip

The proof of the following Lemma \ref{l:commutator_pressure} can be found in \cite{D_JDE} by Danchin. 

\begin{lemma}\label{l:commutator_pressure}
Let $(s,p,r)$ satisfy condition \eqref{Lip_cond} and $\sigma$ be in $]-1,\, s-1]$. Assume that $w\in B^{\sigma}_{p,r}$ and $A$ is a bounded function on $\R^d$ such that $\nabla A\in B^{s-1}_{p,r}$. Then, there exists a constant $C=C(s,p,r,\sigma ,d)$ such that for all $i\in \{1, \dots, d\}$, we have: 
\begin{equation*}
\|\d_i[A, \Delta_j]w\|_{L^p}\leq C\, c_j \, 2^{-j\sigma}\|\nabla A\|_{B^{s-1}_{p,r}}\|w\|_{B^{\sigma}_{p,r}}\quad \text{for all }j\geq -1\, ,
\end{equation*}
with $\|(c_j)_{j\geq -1}\|_{\ell^r}=1$.
\end{lemma}

The next statement concerns a standard commutator estimate between the transport operator and the frequency localisation operator.  
\begin{lemma}\label{l:commutator_est}
Assume that $v\in B^s_{p ,r}$ with $(s,p,r)$ satisfying the Lipschitz condition \eqref{Lip_cond}. Denote by $[v\cdot \nabla ,\Delta_j]f=(v\cdot \nabla)\Delta_j f-\Delta_j(v\cdot \nabla) f$ the commutator between the transport operator $v\cdot \nabla$ and the frequency localisation operator $\Delta_j$. Then, for every $f\in B^{s}_{p,r}$, 
\begin{equation*}
\left\|\left(2^{js}\|[v\cdot \nabla ,\Delta_j]f\|_{L^p}\right)_j\right\|_{\ell^r}\leq C\left(\|\nabla v\|_{L^\infty}\|f\|_{B^{s}_{p ,r}}+\|\nabla v\|_{B^{s-1}_{p ,r}}\|\nabla f\|_{L^\infty}\right)
\end{equation*}
and also, for every $f\in B^{s-1}_{p ,r}$, 
\begin{equation*}
\left\|\left(2^{j(s-1)}\|[v\cdot \nabla ,\Delta_j]f\|_{L^p}\right)_j\right\|_{\ell^r}\leq C\left(\|\nabla v\|_{L^\infty}\|f\|_{B^{s-1}_{p ,r}}+\|\nabla v\|_{B^{s-1}_{p ,r}}\| f\|_{L^\infty}\right)\, ,
\end{equation*}
for some constant $C=C(d,s,p)>0$.
\end{lemma}

Finally, the next result deals with commutators between paraproduct operators and Fourier multipliers.
\begin{lemma}
Let $\kappa$ be a smooth function on $\R^d$, which is homogeneous of degree $m$ away from a neighborhood of $0$. Take $(s,p,r)\in \R \times [1,+\infty]^2$ and $v$ a vector field such that $\nabla v\in L^\infty$. Then, for every $f\in B^s_{p ,r}$, one has
\begin{equation*}
\left\|[T_v,\kappa (D)]f\right\|_{B^{s-m+1}_{p ,r}}\leq C(d,s)\, \|\nabla v\|_{L^\infty}\|f\|_{B^s_{p ,r}}\, .
\end{equation*}
\end{lemma}

\subsection{Transport equations}
In this paragraph, we deal with the transport equations in non-homogeneous Besov spaces. We refer to Chapter 3 of \cite{B-C-D} for additional details. We study, in $\R_+\times \R^d$, the initial value problem 
\begin{equation}\label{general_transport}
\begin{cases}
\d_t f+v\cdot \nabla f =g\\
f_{|t=0}=f_0\, .
\end{cases}
\end{equation}
We always assume the velocity field $v=v(t,x)$ to be a Lipschitz divergence-free function. 
In the case when the Lipschitz condition \eqref{Lip_cond} is satisfied, we have the embedding $B^s_{p ,r}\hookrightarrow W^{1,\infty}$.

We state now the main well-posedness result concerning problem \eqref{general_transport} in Besov spaces. We point out also that the notation $C^0_w([0,T]; X)$, with $X$ a Banach space, refers to the space of functions which are continuous in time with values in $X$ endowed with its weak topology.

\begin{theorem}\label{thm_transport}
Let $(s,p,r)\in \R\times [1,+\infty ]^2$ satisfies the Lipschitz condition \eqref{Lip_cond}. Given $T>0$, take $g\in L^1_T(B^s_{p ,r})$. Assume that $v\in L^1_T(B^s_{p ,r})$ and that there exist two real numbers $q>1$ and $\sigma >0$ such that $v\in L^q_T(B^{-\sigma}_{\infty ,\infty})$. Finally, let $f_0\in B^s_{p ,r}$ be the initial datum. Then, the transport equation \eqref{general_transport} has a unique solution $f$ in:
\begin{itemize}
\item the space $C^0\left([0,T];B^s_{p ,r}\right)$ if $r<+\infty$;
\item the space $\left(\displaystyle \bigcap_{s^\prime <s} C^0\left([0,T];B^{s^\prime}_{p ,\infty}\right)\right) \cap C^0_w\left([0,T];B^s_{p ,\infty}\right)$ if $r=+\infty$.
\end{itemize}
Moreover, the unique solution satisfies the following estimate:
\begin{equation*}
\|f\|_{L^\infty_T(B^s_{p ,r})}\leq \exp \left(C\int_0^T\|\nabla v\|_{B^{s-1}_{p ,r}}\, \detau \right)\left(\|f_0 \|_{B^{s}_{p ,r}}+\int_0^T\exp \left(-C\int_0^t\|\nabla v\|_{B^{s-1}_{p ,r}}\, \detau \right)\|g(t)\|_{B^{s}_{p ,r}}\, \dt\right),
\end{equation*}
for some constant $C=C(d,s,p,r)>0$.
\end{theorem}


To conclude this paragraph, let us show a refinement of Theorem \ref{thm_transport}, proved by Vishik in \cite{Vis} and in a different way by Hmidi and Keraani (see \cite{H-K}). It states that, if $\div v=0$ and the Besov regularity index is $s=0$, the estimate in Theorem \ref{thm_transport} can be replaced by an inequality which is \textit{linear} with respect to $\|\nabla v\|_{L^1_T(L^\infty)}$.

\begin{theorem}\label{thm:improved_est_transport}
Given $T>0$, assume that $v$ is a  divergence-free vector field such that $\nabla v\in L^1_T(L^\infty )$ and let $g\in L^1_T(B^0_{\infty ,r})$. Take $r\in [1,+\infty ]$ and $f_0\in L^1_T(B^0_{\infty ,r})$. Then, there exists a constant $C=C(d)$ such that, for any solution $f$ to problem \eqref{general_transport} in
$C^0([0,T]; B^0_{\infty ,r})$ (or with the usual modification of $C^0$ into $C^0_w$ if $r=+\infty$), we have 
\begin{equation}
\|f\|_{L^\infty_T(B^0_{\infty ,r})}\leq C\left(\|f_0\|_{B^0_{\infty ,r}}+\|g\|_{L^1_T(B^0_{\infty ,r})}\right)\left(1+\int_0^T\|\nabla v(\tau )\|_{L^\infty}\, \detau \right)\, .
\end{equation}
\end{theorem}


\end{document}